\documentclass[AMA,STIX1COL]{WileyNJD-v2}

\articletype{Article Type}%

\received{}
\revised{}
\accepted{}

\newtheoremstyle{plain}
{\topsep}   % ABOVESPACE
{\topsep}   % BELOWSPACE
{\itshape}  % BODYFONT
{0pt}       % INDENT (empty value is the same as 0pt)
{\bfseries} % HEADFONT
{.}         % HEADPUNCT
{5pt plus 1pt minus 1pt} % HEADSPACE
{}          % CUSTOM-HEAD-SPEC

\theoremstyle{plain}
\newtheorem{thm}{\bf Theorem}[section]
\newtheorem{lem}[thm]{\bf Lemma} 

\newtheorem{pro}[thm]{\bf Proposition}

\theoremstyle{definition}

\newtheorem{defn}[thm]{\bf Definition}
\newtheorem{exa}[thm]{\bf Example}

\theoremstyle{remark}
\newtheorem{rem}{\bf Remark}

\newcommand{\rk}{{\rm rank\,}}
\newcommand{\im}{{\rm Im\,}}
\newcommand{\Span}{{\rm span\,}}
\newcommand{\rd}{{\rm d}}
 
\usepackage{mathtools} 
\usepackage{tikz}
\usepackage{float}
\usetikzlibrary{shapes.geometric, arrows,patterns}
\usetikzlibrary{matrix}
\begin{document}

\title{Feedback linearization  of nonlinear differential-algebraic control systems \protect\thanks{The first author is currently supported by Vidi-grant 639.032.733.} }

\author[1]{Yahao Chen*} 

%\author[2]{Witold Respondek}

\address[1]{Bernoulli Institute for Mathematics, Computer Science, and
	Artificial Intelligence, University of Groningen, The Netherlands.  } 
%\address[2]{Normandie Universit\'{e}, INSA-Rouen, LMI, 76801 Saint-Etienne-du-Rouvray, France.}

\corres{Email: yahao.chen@rug.nl.}

\abstract[Summary]{In this paper, we study feedback linearization problems for nonlinear differential-algebraic control systems (DACSs). We consider two kinds of feedback equivalences, namely,  external feedback equivalence, which is defined (locally) on   the  whole generalized state space, and   internal feedback equivalence, which is defined  on  the locally maximal controlled invariant submanifold (i.e., on the  set where   solutions  exist).   Necessary and sufficient conditions are given for    locally internal and the locally external feedback linearizability of DACSs with the help of a   notion  called   explicitation with driving variables, which attaches a class of ordinary differential equation control systems (ODECSs) to a given DACS. We show that  the feedback linearizability of a DACS  is closely related to  the   involutivity of the  linearizability distributions for the explicitation systems. Finally, we apply our results of feedback linearization of DACSs to both academical and practical systems.}

%\jnlcitation{\cname{%
%\author{Williams K.}, 
%\author{B. Hoskins}, 
%\author{R. Lee}, 
%\author{G. Masato}, and 
%\author{T. Woollings}} (\cyear{2016}), 
%\ctitle{A regime analysis of Atlantic winter jet variability applied to evaluate HadGEM3-GC2}, \cjournal{Q.J.R. Meteorol. Soc.}, \cvol{2017;00:1--6}.}

\maketitle

\section{Introduction}\label{Sec:intro}
Consider a nonlinear  differential-algebraic control system (DACS) of the form
\begin{align}\label{Eq:DACS0}
	\Xi^u:E(x)\dot x = F(x)+G(x)u,
\end{align}
where $x\in X$ is called the generalized state and $(x,\dot x)\in TX$, where $TX$ is the tangent bundle of an open subset $X$ in $\mathbb R^n$ (or, more general, of an $n$-dimensional smooth manifold $X$), and $u\in \mathbb R^m$ is the vector of  inputs, and where $E:TX\rightarrow\mathbb R^l$, $F: X\rightarrow\mathbb R^l$ and $G:X\rightarrow \mathbb R^{l\times m}$ are smooth  maps.
The word ``smooth'' will always mean $\mathcal C^{\infty}$-smooth throughout the paper. We denote a DACS of the form (\ref{Eq:DACS0})   by $\Xi^u_{l,n,m}=(E,F,G)$  or, simply, $\Xi^u$.  
%We call  $x$ in (\ref{Eq:DACS0})  the ``generalized'' state because it is different from the state of  a classical ODE control system ODECS , which is
%\begin{align}\label{Eq:ODEs}
%{\dot x = f(x)+\sum_{i=1}^{m}g_i(x)u_i,}
%\end{align}
%{where {$f,g_1,\dots,g_m:X\rightarrow TX$}}. Note that the variables {of} the ``generalized'' states play {two} different roles for the system.  More specifically,  non-invertibility of $E(x)$ may imply the existence of algebraic constraints {and} some   variables of $x$ (even some $u$-variables) are constrained by the algebraic constraints. {On the other hand}, some {other} variables of $x$ are free and they {play} the role of an input (since they enter the system statically). Note that although the free variables of $x$ may perform "like" inputs, we will {emphasize} their differences with the original control input $u$. 
%The vector $u$ is a predefined control input but $x$ is a predefined ``state'' when modeling the system,  unless the analysis shows the properties of the variables in $x$ and $u$, we do not know their actual roles in the system. 
A linear DACS is of the form
\begin{align}\label{Eq:linDAEcontol}
	\Delta^u:E\dot x=Hx+Lu,
\end{align} 
where $E,H\in \mathbb R^{l\times n}$ and $L\in \mathbb R^{l\times m}$. Denote a linear DACS by $\Delta^u_{l,n,m}=(E,H,L)$ or, simply, $\Delta^u$. Linear DACSs have been studied for decades, there is a rich literature devoted to them (see, e.g., the surveys \cite{lewis1986survey,lewis1992tutorial} and textbook \cite{dai1989singular}). In the context of this paper, we will need results about canonical forms \cite{loiseau1991feedback,lebret1994proportional,chen2021geometric}, controllability \cite{berger2013controllability,cobb1984controllability,frankowska1990controllability}, and geometric subspaces \cite{geerts1993invariant,ozccaldiran1986geometric}. The motivation of studying linear and nonlinear DACSs is their frequent presence in mathematical models of practical systems as constrained mechanics \cite{rabier2000nonholonomic}, chemical processes \cite{Kumar1995}, electrical circuits \cite{riaza2008differential}, etc. 

The map $E$ of a DACS   (\ref{Eq:DACS0}) is not necessarily square (i.e., $l\neq n$) nor invertible. As a consequence,  some free variables and constrained variables can be implicitly present in the generalized state $x$ (and also some constrained control variables can exist in the input $u$). We have proposed two normal forms     to distinguish the different roles of variables for nonlinear DACSs in \cite{chen2021nfs}.
It is noted that although the free variables of $x$ may perform like an input, we will distinguish them from the real active control variables $u$.  The control $u$ can be changed  physically and actively via some actuators while the free variables in $x$ are states coming from unknown constrained forces (e.g., the friction force $F_f$ in Example \ref{Ex:1} below) or some redundancies of mathematical modeling (e.g., the Lagrange multipliers when modeling constrained mechanical systems  \cite{rabier2000nonholonomic}).  
In the case of $E(x)=I_n$, the DACS  (\ref{Eq:DACS0}) becomes  an  ordinary differential equation control system (ODECS)
\begin{align}\label{Eq:ODEs}
	\dot x = f(x)+\sum_{i=1}^{m}g_i(x)u_i,
\end{align}
where $f=F$ and $g_i$, $1\le i\le m$, being the columns of $G$, become vector fields on $X$. The feedback linearization problem for nonlinear ODECSs (i.e., when there exist a local change of coordinates in the state space and a feedback transformation such that the transformed system has a linear form in the new coordinates) has drawn the attention of researchers for decades (e.g. see survey papers \cite{respondek1985geometric,tall2005feedback} and books \cite{nijmeijer1990nonlinear,Isidori:1995:NCS:545735}). The solution of the feedback linearization problem of ODECSs was first given in Brockett's paper \cite{brockett1978feedback} and developed by Jakubczyk and Respondek \cite{Jakubczyk&Respondek1980}, Su \cite{SU198248}, Hunt et Su \cite{hunt1983global}. Compared to the ODECSs, fewer results on the linearization problems of DACSs can be found.  Xiaoping \cite{4793293} transformed a nonlinear DACS into a linear one by state space transformations, Kawaji \cite{kawaji1994feedback} gave sufficient conditions for the feedback linearization of a special class of DACSs, Wang and Chen  \cite{983833} considered a semi-explicit differential-algebraic equation (DAE) and linearized the differential part of the DAE. The linearization of semi-explicit DAEs under equivalence of different levels  is studies in \cite{chen2019internal}.

In the present paper, our purpose is to find when a given DACS of the form (\ref{Eq:DACS0}) is locally equivalent to a linear completely controllable one (see the definition of the complete controllability of linear DACSs in  \cite{berger2013controllability}). In particular, we will consider two kinds of equivalence relations, namely, the external feedback equivalence given in Definition \ref{Def:ex-fb-equi} and the internal feedback equivalence given in Definition \ref{Def:in-fb-equi}. Note that the words ``external'' and ``internal'', appearing throughout this paper, basically mean that we consider the DACS on an open neighborhood of the generalized state space $X$ and on the \emph{locally maximal controlled invariant submanifold} $M^*$ (see Definition \ref{Def:CIS}),  respectively. We have discussed in detail  the differences and relations of the two equivalence relations for linear DAEs   \cite{chen2021geometric}, and for semi-explicit DAEs  \cite{chen2019internal}.  We will use a notion  called the \emph{explicitation with driving variables} (see Definition \ref{Def:explicitation}, firstly proposed in \cite{chen2021from} for linear DACSs) to  connect  nonlinear DACSs with nonlinear ODECSs. Via the explicitation with driving variables, we can interpret the linearizability of a DACS under internal or external feedback equivalence as that of an explicitation system under system feedback equivalence (see Definition \ref{Def:sys-fb}).   
%A previous version of the results in the present paper can be consulted in Chapter 5 of the thesis \cite{chen2019geometric}. In particular, compared to  \cite{chen2019geometric},  the proof of Theorem \ref{Thm:IFL} in the present paper  is different since we use a new Lemma \ref{cl:1} to simply the proof, and we add an Example \ref{Ex:1}  to illustrate the internal feedback linearization of DACSs.

%The other contribution is to propose \cyan{two normal forms} based on the notion of {maximal controlled invariant submanifold}. These normal forms are helpful {in understanding} the {role} of the variables in a DACS, e.g., to see which variables of the ``generalized'' state are actually free and which control inputs are actually constrained by  algebraic constraints.

The paper is organized as follows: In Section \ref{Sec:2}, we define the external and the internal feedback equivalences and discuss their relations with solutions.  In Section \ref{Sec:3}, we use the notion of explicitation with driving variables to connect DACSs with ODECSs. Necessary and sufficient conditions for both the external and the internal feedback linearization problems of DACSs are given in Section \ref{Sec:4}.   We illustrate the results of Section~\ref{Sec:4}  by the two examples in Section \ref{Sec:5}.    The conclusions and  perspectives of this paper are given in Section~\ref{Sec:6} and a technical proof is given in Appendix. 

\section{External and internal feedback equivalence}\label{Sec:2} 
We use the following notations in the present paper: 
We denote by $T_x M\in \mathbb R^n$  the tangent space at $x\in M$ of a differentiable submanifold $M$ of $\mathbb{R}^n$.
We use $\mathcal C^k$ to denote the class of {$k$-times differentiable functions} and $GL(n, \mathbb R)$ to denote the group of nonsingular matrices of $\mathbb R^{n\times n}$.  For a smooth map $f:X\to \mathbb R$, we  denote its differential by ${\rm d} f=\sum^n_{i=1}\frac{\partial f}{\partial x_i}{\rm d}x_i=[\frac{\partial f}{\partial x_1},\dots,\frac{\partial f}{\partial x_n}]$. For a map $A:X\to \mathbb R^{m\times n}$, $\ker A(x)$, ${\rm Im\,} A(x)$ and ${\rm rank\,}A(x)$ are the kernel, the image and the rank of $A$ at $x$, respectively.  For a full row rank map $R:X\to \mathbb R^{r\times n}$, we denote by $R^{\dagger}:X\to \mathbb R^{n\times r}$ the right inverse of $R$, i.e., $RR^{\dagger}=I_r$.  For two column vectors  $v_1\in \mathbb R^m$ and $ v_2\in \mathbb R^n$, we  write $(v_1,v_2)=[v^T_1,v^T_2]^T\in \mathbb R^{m+n}$.
We assume  the reader is familiar with  basic notions of differential geometry such  as smooth  embedded submanifolds,  involutive distributions and refer the reader e.g. to the book  \cite{lee2001introduction} for the formal definitions of such notions. 
\begin{defn}[solutions and admissible set]
	For a DACS $\Xi^u_{l,n,m}=(E,F,G)$,  a curve $(x,u):I\rightarrow X\times\mathbb R^m$  defined on an open interval $I\subseteq \mathbb R$ with $x(\cdot)\in\mathcal{C}^{1}$  and $u(\cdot)\in\mathcal{C}^{0}$, is called a solution of $\Xi^u$  if for all $t\in I$, $E(x(t))\dot x(t)=F(x(t))+G(x(t))u(t)$. We call a point $x_a\in X$ \emph{admissible} if there exists at least one solution $(x(\cdot),u(\cdot))$ such that $x(t_a)=x_a$ for a certain $t_a\in I$. The set of all admissible points will be called the admissible set (or the consistency set) of $\Xi^u$ and denoted by $S_a$.
\end{defn} 
A smooth connected embedded submanifold $M$  is called   \emph{controlled invariant}     if   for any point $x_0\in M$, there exists  a solution $(x,u):I\rightarrow M\times \mathbb R^m$ such that $x(t_0)=x_0$ for a certain $t_0\in I$ and $x(t)\in M$, $\forall \, t\in I$. Fix an admissible  point $x_a\in X$, a smooth connected embedded submanifold $M$ containing $x_a$ is called \emph{locally controlled invariant} if there exists a neighborhood $U$ of $x_a$ such that $M\cap U$ is controlled invariant.
\begin{defn}[locally maximal controlled invariant submanifold]  \label{Def:CIS}
	A locally controlled invariant submanifold $M^*$, around an admissible point $x_a$, is called \emph{maximal} if there exists a neighborhood $U$ of $x_a$ such that for any other locally controlled invariant submanifold $M$, we have $M\cap U \subseteq M^*\cap U$.
\end{defn}
The locally maximal controlled invariant submanifold $M^*$   of a DACS  can be constructed via the following \emph{geometric reduction method}, which was frequently used  (see e.g., \cite{reich1991existence,rabier1994geometric,riaza2008differential,berger2016controlled,chen2021nfs}) for studying   existence  of solutions for DAEs and DACSs. 
\begin{defn}[geometric reduction method \cite{chen2021nfs,berger2016controlled}]\label{Def:gr}
	For a DACS $\Xi^u_{l,n,m}=(E,F,G)$, fix a point $x_p\in X$. Let $U_0$ be a connected subset of $X$ containing $x_p$. Step 0: Set $M_0=X$ and $M^c_0=U_0$. Step $k$ ($k>0$): Suppose that  a sequence of smooth connected embedded submanifolds $M^c_{k-1}\subsetneq \cdots\subsetneq M^c_0$ of $U_{k-1}$ for a certain $k-1$, have been constructed. Define recursively 
	\begin{align*}
		{M_{k}} := \left\{ {x \in M^c_{k-1}\,|\,F(x) \in E(x){T_x M^c_{k-1}}}+{\rm Im\,} G(x) \right\}.
	\end{align*} 
	As long as $x_p\in M_k$, let $M^c_k=M_k\cap U_k$ be a smooth embedded connected submanifold for some neighborhood $U_k\subseteq U_{k-1}$ of $x_p$.
\end{defn} 
\begin{pro}[\cite{chen2021nfs}]\label{Pro:M^*}
	In the above geometric reduction method, there always exists  a smallest $k^*$ such that either $k^*$ is the smallest integer for which $x_p\notin M_{k^*+1}$ or $k^*$ is the smallest integer such that $x_p\in M^c_{k^*+1}$ and $M^c_{k^*+1}\cap U_{k^*+1}=M^c_{k^*}\cap U_{k^*+1}$. In the latter case, denote $M^*=M^c_{k^*+1}$ and
	assume that there exists an open neighborhood $U^*\subseteq U_{k^*+1}$ of $x_p$ such that    $\dim E(x)T_xM^*=const.$ and  $E(x)T_xM^*+{\rm Im\,}G(x)=const.$ for all $x\in M^*\cap U^*$, then
	\begin{itemize} 
		\item [(i)]   $x_p$ is an admissible point, i.e., $x_p=x_a$ and $M^*$ is the locally maximal controlled invariant submanifold around $x_p$; 
		\item [(ii)]     $M^*$ coincides locally with the admissible set $S_a$, i.e., $M^*\cap U^*=S_a\cap U^*$.
	\end{itemize}
\end{pro}
By item (ii) of Proposition \ref{Pro:M^*}, the admissible set $S_a$ locally coincides with $M^*$ on the neighborhood $U^*$ of $x_p$. So   any point $x_0\in U^*\backslash M^*$ is not admissible and there exist   no 
solutions passing through $x_0$. Thus to study solutions of a DACS,   it is  convenient to consider only the restriction of the DACS to  its locally maximal controlled invariant  submanifold $M^*$. We have shown how to restrict a DACS   to the submanifold $M^*$ in Remark 3.4(iv) and Theorem 4.4(i) of \cite{chen2021nfs} with the help of normal forms, now we define formally the notion of restriction as follows. 

Consider a DACS $\Xi^u_{l,n,m}=(E,F,G)$ and fix an admissible point $x_a\in X$. 
Let $M^*$ be the $n^*$-dimensional maximal controlled invariant submanifold of $\Xi^u$ around $x_a$. Assume that there exists a neighborhood $U$ of $x_a$ such that  for all $x\in M^*\cap U$, 
\begin{enumerate} 
	\item[\textbf{(CR)}]     $\dim E(x)T_xM^*=const.=r^*$ and  $E(x)T_xM^*+{\rm Im\,}G(x)=const.=r^*+(m-m^*)$.
\end{enumerate}
Let $\psi: U\to \mathbb R^n$ be a local diffeomorphism and $z=\psi(x)=(z_1,z_2)$ be local coordinates on $U$ such that $M^*\cap U=\left\lbrace z_2=0 \right\rbrace$, thus $z_1$ are local coordinates on $M^*\cap U$. Then in the new $z$-coordinates, the DACS $\Xi^u$ becomes a system $\tilde \Xi^u_{l,n,m}=(\tilde E,\tilde F,\tilde G)$, given by
$$
\left[ \begin{matrix}
	\tilde E_1(z_1,z_2)&\tilde E_2(z_1,z_2)
\end{matrix}\right]\left[\begin{matrix}
	\dot z_1\\\dot z_2
\end{matrix} \right]=\tilde F(z_1,z_2)+ \tilde G(z_1,z_2)u,
$$
where $\tilde E_1:U\to \mathbb R^{l\times n^*}$, $\tilde E_2:U\to \mathbb R^{l\times (n-n^*)}$, 
$\tilde E\circ\psi=\left[ \begin{matrix}
	\tilde E_1\circ\psi&\tilde E_2\circ\psi
\end{matrix}\right]  =E\cdot\left( \frac{{\partial \psi }}{{\partial x}}\right) ^{-1}$,  $\tilde F\circ\psi= F$ and $\tilde G\circ\psi=G$. Set $z_2=0$ to have the following system (which is defined on $M^*$)
\begin{align}\label{Eq:DACS}
	\left[ \begin{matrix}
		\tilde E_1(z_1,0)&\tilde E_2(z_1,0)
	\end{matrix}\right]\left[\begin{matrix}
		\dot z_1\\0
	\end{matrix} \right]=\tilde F(z_1,0)+ \tilde G(z_1,0)u.
\end{align} By  \textbf{(CR)}, there exist a neighborhood $U_1\subseteq U$ of $x_a$ and  $Q:M^*\cap U_1\rightarrow GL(l,\mathbb R)$ such that   $\tilde E_1^1(z_1)$ and $\tilde G_2(z_1)$ below are of full row rank,
\begin{align*}
	Q(z_1)\left[ {\begin{matrix}
			\tilde E_1(z_1,0)&\tilde  F(z_1,0)&\tilde  G(z_1,0)
	\end{matrix}} \right]=\left[ \begin{matrix}
		\tilde	E^1_1(z_1)&\tilde  F_1(z_1)&\tilde G_1(z_1)\\0&\tilde  F_2(z_1)&\tilde G_2(z_1)\\0&\tilde  F_3(z_1)&0
	\end{matrix} \right],
\end{align*}
where $\tilde E^1_1$,  $\tilde G_2$ are smooth functions defined on $M^*\cap U_1$  with values in $\mathbb R^{r^*\times n^*}$ and $\mathbb R^{(m-m^*)\times m}$, respectively, and $\tilde F_1$, $\tilde F_2$, $\tilde F_3$  and $\tilde G_1$ are matrix-valued functions of appropriate sizes.  Since $\tilde G_2(z_1)$ is of full row rank,  we can always assume 
$\left[ {\begin{smallmatrix}
		\tilde G_1(z_1)\\ \tilde G_2(z_1)
\end{smallmatrix}} \right]=\left[ \begin{smallmatrix}
	\tilde G^1_1(z_1)&\tilde G_1^2(z_1)\\\tilde G^1_2(z_1)&\tilde G_2^2(z_1)
\end{smallmatrix} \right]$ with  $\tilde G_2^2: M^*\cap U_1\rightarrow GL(m-m^*,\mathbb R)$  (if not, we  permute the components of $u$ such that $\tilde G_2^2(z_1)$ is invertible), where $\tilde G^1_1$, $\tilde G^2_1$ and $\tilde G^1_2$ are of appropriate sizes.   Thus, via $Q$ and  the following feedback transformation (note that $a^u,b^u$ are defined on $M^*$ and $b^u(z_1)$ is invertible),
\begin{align*}
	\left[ {\begin{matrix}
			u_1\\u_2
	\end{matrix}} \right]=a^u(z_1)+b^u(z_1)u=\left[ {\begin{matrix}
			0\\
			\tilde F_2(z_1)
	\end{matrix}} \right]+\left[ {\begin{matrix}
			I_{m^*}&0\\
			\tilde G_2^1(z_1)&\tilde G_2^2(z_1)
	\end{matrix}} \right] u,
\end{align*}
the DACS (\ref{Eq:DACS}) is transformed  into
\begin{align}\label{Eq:inDACS}
	\left[ {\begin{matrix}
			\bar E^1_1(z_1)\\0\\0
	\end{matrix}} \right]\dot z_1=\left[ {\begin{matrix}
			\bar F_1(z_1)\\0\\ \bar F_3(z_1)
	\end{matrix}} \right]+\left[ {\begin{matrix}
			{\bar G^1_1(z_1)}& \bar G_1^2(z_1)\\0&I_{m-m^*}\\0&0
	\end{matrix}} \right]\left[ {\begin{matrix}
			u_1\\u_2
	\end{matrix}} \right],
\end{align} 
where $\bar E^1_1= \tilde E^1_1$, $\bar F_3=\tilde F_3$, $\bar F_1=\tilde F_1-\tilde G_1^2(\tilde G_2^2)^{-1} \tilde F_2$, $\bar G_1^1=\tilde G_1^1-\tilde G_1^2(\tilde G_2^2)^{-1}\tilde G^1_2$ and  $\bar G_1^2=\tilde G_1^2(\tilde G_2^2)^{-1}$. 
\begin{defn}[restriction]\label{Def:restrictionDACS}
	The local $M^*$-restriction of $\Xi^u$, denoted by $\Xi^u|_{M^*}$, is given by 
	\begin{align}\label{Eq:restr}
		\Xi^u|_{M^*}=\Xi^{u^*}:E^*(z^*)\dot z^*= F^*(z^*)+ G^*(z^*)u^*.
	\end{align}
	where $z^*=z_1$, $u^*=u_1$, $E^*=\bar E^1_1:M^*\to \mathbb R^{r^*\times n^*}$, $F^*=\bar F_1:M^*\to \mathbb R^{r^*}$ and $G^*=\bar G^1_1:M^*\to \mathbb R^{r^*\times m^*}$ come from    (\ref{Eq:inDACS}), and where  the map $E^*$ is of full row rank $r^*$.
\end{defn} 
\begin{rem}\label{Rem:restri}
	The restriction $\Xi^u|_{M^*}$ is a DACS   of the form (\ref{Eq:DACS0}) with associated dimensions $r^*,n^*,m^*$, i.e., $\Xi^u|_{M^*}= \Xi^{u^*}_{r^*,n^*,m^*}$.   
	It is important to know that  $\Xi^u$ and  $\Xi^u|_{M^*}$ has isomorphic solutions (see Theorem 4.4(i) of \cite{chen2021nfs}). More specifically, a curve $(x(\cdot),u(\cdot))$ is a solution of $\Xi^u$ passing through a point $x_0\in X$  if and only if $(z^*(\cdot),u^*(\cdot))$ is a solution of $\Xi^u|_{M^*}$  passing through $z^*_{0}\in M^*$, where $(z^*(\cdot),0)=\psi(x(\cdot))$, $(z^*_{0},0)=\psi(x_0)$ and $(u^*(\cdot),0)=a^u(z^*(\cdot))+b^u(z^*(\cdot))u(\cdot)$.
\end{rem}
Now we define the external and the internal feedback equivalences for  nonlinear DACSs and compare them by discussing their relations with   solutions.
\begin{defn}[external feedback equivalence]\label{Def:ex-fb-equi}
	Two DACSs $\Xi^u_{l,n,m}=(E,F,G)$ and $\tilde \Xi^{\tilde u}_{l,n,m}=(\tilde E,\tilde F,\tilde G)$ defined on $X$ and $\tilde X$, respectively, are called externally feedback equivalent, shortly ex-fb-equivalent,  if there exist a diffeomorphism $\psi:X\rightarrow \tilde X$ and smooth functions $Q:X\rightarrow GL(l,\mathbb{R})$, $\alpha^u: X\rightarrow \mathbb R^{m}$, $\beta^u: X\rightarrow GL(m,\mathbb R)$ such that
	\begin{align}\label{Eq:fb-equi}
		\begin{array}{c}
			\tilde E(\psi (x))=Q(x)E(x)\left( \frac{{\partial \psi (x)}}{{\partial x}}\right) ^{-1}, \quad \tilde F(\psi (x))=Q(x)\left( F(x)+G(x)\alpha^u(x)\right), \quad \tilde G(\psi(x))=Q(x)G(x)\beta^u(x).
		\end{array}
	\end{align}
	The ex-fb-equivalence of two DACSs $\Xi^u$ and $\tilde \Xi^{\tilde u}$ is denoted by $\Xi^u\mathop  \sim \limits^{ex-fb} \tilde \Xi^{\tilde u}$. If  $\psi: U\rightarrow  \tilde U$ is a local diffeomorphism between neighborhoods $U$ of a point $x_p$ and $\tilde U$ of a point $\tilde x_p=\psi(x_p)$, and $Q(x)$, $\alpha^u(x)$, $\beta^u(x)$ are defined on $U$, we will talk about local ex-fb-equivalence.
\end{defn} 
\begin{defn}[internal feedback equivalence]\label{Def:in-fb-equi}
	Consider two DACSs  $\Xi^u=(E,F,G)$ and $\tilde \Xi^{\tilde u}=(\tilde E,\tilde F,\tilde G)$ defined on $X$ and $\tilde X$, respectively. Fix two admissible points $x_a\in X$ and $\tilde x_a\in \tilde X$.  Assume that
	\begin{enumerate} 
		\item[(A1)]  $ M^*$ and $\tilde M^*$ are  locally maximal controlled invariant submanifolds of $\Xi^u$ around $x_a$ and of $\tilde \Xi^{\tilde u}$ around $\tilde x_a$, respectively.
		\item[(A2)] $M^*$ and $\tilde M^*$ satisfy the constant rank condition \textbf{(CR)} around $x_a$ and $\tilde x_a$, respectively. 
	\end{enumerate}
	Then, $\Xi^u$ and $\tilde \Xi^{\tilde u}$ are called internally feedback equivalent, shortly in-fb-equivalent,  if their restrictions $ \Xi^u|_{M^*}$ and $\tilde \Xi^{\tilde u}|_{\tilde M^*}$  are ex-fb-equivalent. We will denote the in-fb-equivalence of two DACSs by $ \Xi^u\mathop  \sim \limits^{in-fb} \tilde \Xi^{\tilde u}$. 
\end{defn}
\begin{rem}\label{Rem:in-fb-eq}  
	The dimensions of two in-fb-equivalent DACSs
	$\Xi^u$ and $\tilde \Xi^{\tilde u}$ are not necessarily the same. However, since  $\Xi^u|_{M^*}=\Xi^{u^*}_{l^*,n^*,m^*}$   and $\tilde \Xi^{\tilde u}|_{\tilde M^*}=\tilde \Xi^{\tilde u^*}_{\tilde l^*,\tilde n^*,\tilde m^*}$    are required to be external feedback equivalent, their dimensions have to be the same, i.e., $r^*=\tilde r^*$, $n^*=\tilde n^*$ and $m^*=\tilde m^*$.
\end{rem}
Both the ex-fb-equivalence and the in-fb-equivalence preserve  solutions of DACSs. Indeed, consider two ex-fb-equivalent DACSs $\Xi^u$ and $\tilde \Xi^{\tilde u}$,   the diffeomorphism $\tilde x =\psi(x)$ and the feedback transformation $u= \alpha^u(x)+\beta^u(x)\tilde u$ (defined on $X$)  establish a  one to one correspondence between  solutions $(x,u)$ of $\Xi^u$ and   solutions $(\tilde x,\tilde u)$ of $\tilde \Xi^{\tilde u}$, i.e., $\tilde x=\psi (x)$ and $u=\alpha^u(x)+\beta^u(x) \tilde u$. For two in-fb-equivalent DACSs $\Xi^u$ and $\tilde \Xi^{\tilde u}$, by $\Xi^u|_{M^*}\overset{ex-fb}{\sim}\tilde\Xi^{\tilde u}|_{\tilde M^*}$, there exist  a diffeomorphism $\tilde z^*=\psi^*(z^*)$ between $M^*$ and $\tilde M^*$, and a feedback transformation  $u^*=\alpha^{u^*}(z^*)+\beta^{u^*}(z^*) \tilde u^*$ defined on $M^*$ mapping solutions $(z^*,u^*)$ of $\Xi^u|_{M^*}$ into  solutions $(\tilde z^*,\tilde u^*)$ of $\tilde \Xi^{\tilde u}|_{\tilde M^*}$. Recall from Remark \ref{Rem:restri} that the DACSs $\Xi^u$ and $\tilde \Xi^{\tilde u}$ have isomorphic solutions with their restrictions $\Xi^u|_{M^*}$ and $\tilde \Xi^{\tilde u}$, respectively.
% solutions $(x,u)$ of $\Xi^u$ are given by
%$x=\psi^{-1}(z^*,0)$ and $u=b^{-1}(z^*)((u^*,0)-a(z^*))$,  and solutions $(\tilde x,\tilde u)$ of $\tilde \Xi^{\tilde u}$ are given by
%$\tilde x=\tilde\psi^{-1}(\tilde z^*,0)$ and $\tilde u=\tilde b^{-1}(\tilde z^*)((\tilde u^*,0)-a(\tilde z^*))$, where $\psi$ and $\tilde \psi$ are local diffeomorphisms defined on $X$ and $\tilde X$, respectively, and $a,b$ and $\tilde a,\tilde b$  define feedback transformations on $M^*$ and $\tilde M^*$, respectively. 
So  solutions $(x,u)$ of $\Xi^u$  are also in a   one-to-one correspondence with solutions $(\tilde x,\tilde u)$  of $\tilde \Xi^{\tilde u}$ if $\Xi^u\overset{in-fb}{\sim}\tilde \Xi^{\tilde u}$.

Conversely,   if  solutions of two DACSs $\Xi^u$ and $\tilde \Xi^{\tilde u}$ are in a one-to-one   correspondence via a diffeomorphism   and a feedback transformation, then the two DACSs are in-fb-equivalent, however, they are \emph{not} necessarily ex-fb-equivalence.  The reason  is that  solutions of DACSs exist  on   maximal controlled invariant submanifolds only, by assuming  two DACSs have corresponding solutions, we only have the information that the two restrictions $\Xi^u|_{M^*}$ and $\tilde \Xi^{\tilde u}|_{\tilde M^*}$ can be transformed into each other via a $Q$-transformation and a feedback  transformation defined on $M^*$, together with a diffeomorphism  between   $M^*$ and $\tilde M^*$, we do not know, however, if those transformations can be extended outside the submanifolds  $M^*$ and $\tilde M^*$. 
\begin{exa}
	Consider two DACSs $\Xi^u_{3,3,1}=(E,F,G)$ defined on $X=\mathbb R^3$ and $\tilde \Xi^{\tilde u}_{3,3,1}=(\tilde E,\tilde F,\tilde G)$ defined on $\tilde X=\mathbb R^3$, where
	$$
	\begin{aligned}
		E(x)=\left[ \begin{matrix}
			1&0&0\\
			0&0&0\\
			0&0&0
		\end{matrix}\right], \ \ F(x)=\left[ \begin{matrix}
			(x_1)^2\\
			e^{x_1}x_2\\
			x_3
		\end{matrix}\right], \ \ G(x)=\left[ \begin{matrix}
			e^{x_2}\\
			0\\
			0
		\end{matrix}\right], \ \  \tilde   E(\tilde x)=\left[ \begin{matrix}
			1&\tilde x_2&0\\
			0&0&0\\
			0&1&0
		\end{matrix}\right],  \ \ \tilde F(\tilde  x)=\left[ \begin{matrix}
			\tilde  x_2\\
			e^{\tilde x_1}\tilde x_2\\
			\tilde x_3
		\end{matrix}\right], \ \ \tilde G(\tilde  x)=\left[ \begin{matrix}
			1\\
			0\\
			0
		\end{matrix}\right].
	\end{aligned}
	$$
	It is seen that $M^*=\{(x_1,x_2,x_3)\in \mathbb R^3\,|\, x_2=x_3=0\}$ and $\tilde M^*=\{(\tilde x_1,\tilde x_2,\tilde x_3)\in \mathbb R^3\,|\, \tilde x_2=\tilde x_3=0\}$. The restrictions $\Xi^u|_{M^*}:\dot x_1=(x_1)^2+u$ and $\tilde \Xi^{\tilde u}|_{\tilde M^*}:\dot {\tilde x}_1=\tilde u$ are ex-fb-equivalent  via $Q(x_1)=1$, $\tilde x_1=\psi(x_1)=x_1$ and $\tilde u=(x_1)^2+u$. Thus we have $\Xi^u\overset{in-fb}{\sim}\tilde \Xi^{\tilde u}$. It is clear that solutions $((x_1,0,0),u)$ of $\Xi^u$ and solutions $((\tilde x_1,0,0),\tilde u)$ of $\tilde \Xi^{\tilde u}$ have a one-to-one correspondence. However, the two DACSs are \emph{not} ex-fb-equivalent since $\rk E(x)\neq \rk \tilde E(\tilde x)$ (the matrix-valued functions  $E(x)$ and $\tilde E(\tilde x)$ of two ex-fb-equivalent DACSs should have the same rank).
\end{exa}
Both the external and the internal feedback equivalences play an important role for DACSs. The internal feedback equivalence is convenient when we are only interested in   solutions  passing  through an admissible point and evolving  on $M^*$. The ex-fb-equivalence is useful when the initial point $x_0\notin M^*$, i.e., $x_0$ is not admissible, then there are no  solutions passing through $x_0$ but there may still exist a jump from the inadmissible point $x_0$ to an  admissible one on $M^*$,  see our recent publication \cite{chen2021ADHS}, where we use  external equivalence   to study  jump solutions of nonlinear DAEs. 
\section{Explicitation of nonlinear differential- algebraic control systems}\label{Sec:3} 
We have proposed the notion of explicitation (with driving variables) for linear DACS in \cite{chen2021from} (or see Chapter 3 of \cite{chen2019geometric}), we now extend this notion to nonlinear DACSs.
\begin{defn}[explicitation with driving variables]\label{Def:explicitation}
	Given a DACS $\Xi^u_{l,n,m}=(E,F,G)$, fix a point $x_p\in X$. Assume that   $\rk E(x)=const.=r$  around $x_p$. Then locally there exists $Q:X\to GL(l,\mathbb R)$ such that $E_1$ of  $Q(x)E(x)=\left[ {\begin{smallmatrix}
			{E_1(x)}\\
			{0}
	\end{smallmatrix}} \right]$ is of full row rank $r$, denote
	$$
	Q(x)F(x)=\left[ {\begin{matrix}
			{F_1(x)}\\
			{F_2(x)}
	\end{matrix}} \right], \ \ Q(x)G(x)=\left[ {\begin{matrix}
			{G_1(x)}\\
			{G_2(x)}
	\end{matrix}}\right].
	$$
	Define locally the maps  $f: X\to \mathbb R^{n}$, $g^u:X\to \mathbb R^{n\times m}$, $g^v:X\to \mathbb R^{n\times s}$, $h:X\to \mathbb R^{p}$, $l^u:X\to \mathbb R^{p\times m}$, where $s=n-r$ and $p=l-r$,  such that
	$$
	\begin{array}{c}
		f(x)=E_1^{\dagger }(x)F_1(x), \ \ \ g^u(x)=E_1^{\dagger }(x)G_1(x), \ \ \ {\rm Im\,} g^v(x)=\ker E_1(x),\ \ \  h(x)=F_2(x), \ \ \ l^u(x)=G_2(x),
	\end{array}	
	$$
	where $E^{\dagger}_1$ is a right inverse of $E_1$. By a $(Q,v)$-explicitation, we will call any ODECS
	\begin{align}\label{Eq:controlsys1}
		\Sigma^{uv}:\left\lbrace \begin{array}{c@{\ }l}
			\dot x &= f(x) +  g^u(x)u+g^v(x)v,\\
			y  &= h(x)+l^u(x)u,
		\end{array} \right.
	\end{align}
	where $v\in \mathbb R^{s\times n}$ is called \emph{the vector of driving variables}. System (\ref{Eq:controlsys1}) is denoted by $\Sigma^{uv}_{n,m,s,p}=(f,g^u,g^v,h,l^u)$ or, simply, $\Sigma^{uv}$. 
\end{defn}
Clearly, in the above definition, the choices of the invertible map $Q$, the right inverse $E^{\dagger}_1$ and the map $g^v$ satisfying $\im g^v=\ker E_1=\ker E$,  are not unique. The following proposition shows that a $(Q,v)$-explicitation of a given DACS $\Xi^u$  is  an ODECS defined up to a feedback transformation, an output multiplication and a generalized output injection, i.e., a class of control systems. Throughout the class of all $(Q,v)$-explicitations of $\Xi^u$ will be called the explicitation class. For a particular ODECS $\Sigma^{uv}$ belonging to the explicitation class $\textbf{Expl}(\Xi^u)$ of $\Xi^u$, we will write  $\Sigma^{uv}\in\mathbf{Expl}(\Xi^u)$.
\begin{pro}\label{Pro:explinonDACS}
	Assume that an ODECS  $\Sigma^{uv}_{n,m,s,p}=(f,g^u,g^v,h,l^u)$ is a $(Q,v)$-explicitation of a DACS $\Xi^u=(E,F,G)$ corresponding to the choice of invertible matrix $Q(x)$, right inverse $E_1^{\dagger}(x)$ and matrix $g^v(x)$.  We have that an ODECS  $\tilde \Sigma^{u,\tilde v}_{n,m,p}=(\tilde f,\tilde g^{u},\tilde g^{\tilde v},\tilde h,\tilde l^{u})$ is a $(\tilde Q,\tilde v)$ -explicitation of  $\Xi^u$ corresponding to the choice of invertible matrix $\tilde Q(x)$, right inverse $\tilde E_1^{\dagger}(x)$ and matrix $\tilde g^{\tilde v}(x)$  if and only if $\Sigma^{uv}$ and $\tilde \Sigma^{u,\tilde v}$ are equivalent via a $v$-feedback transformation of the form $v=\alpha^v(x)+\lambda(x)u+\beta^v(x)\tilde v$, a generalized output injection $\gamma(x)y=\gamma(x)(h(x)+l^u(x)u)$ and an output multiplication $\tilde y=\eta(x)y$, 
	which map 
	$$
	\begin{array}{c}
		f\mapsto \tilde f=f+\gamma h+g^v\alpha^v, \ \ \  g^u\mapsto \tilde g^{u}=g^u+ \gamma l^u+g^v\lambda, \ \ \ g^v\mapsto  \tilde g^{\tilde v}=g^v\beta^v, \ \ \  h\mapsto  \tilde h=\eta h,  \ \  \ l^u\mapsto  \tilde l^u=\eta l^u.
	\end{array}
	$$
	where $\alpha^v(x)$, $\beta^v(x)$, $\gamma(x)$, $\lambda(x)$, $\eta(x)$ are smooth matrix-valued functions, and $\beta^v(x)$ and $\eta(x)$ are invertible. 
\end{pro}
We omit the proof of Proposition \ref{Pro:explinonDACS} since it follows the same line as that of Proposition 2.3 in \cite{chen2021from}. Now we will define an equivalence relation for two ODECSs of the form (\ref{Eq:controlsys1}). 
\begin{defn}[system feedback equivalence]\label{Def:sys-fb}
	Two ODECSs $\Sigma^{uv}_{n,m,s,p}=(f,g^u,g^v,h,l^u)$ and $\tilde \Sigma^{\tilde u \tilde v}_{n,m,s,p}=(\tilde f,\tilde g^{\tilde u},\tilde g^{\tilde v},\tilde h,\tilde l^{\tilde u})$ defined on $X$ and $\tilde X$, respectively, are called  system feedback equivalent, or shortly sys-fb-equivalent, if there exist a  diffeomorphism $\psi:  X\rightarrow \tilde X$, smooth functions $\alpha^u(x)$, $\alpha^v(x)$, $\lambda(x)$ and $\gamma(x)$ with values in $\mathbb R^m$, $\mathbb R^s$, $\mathbb R^{s\times m}$ and $\mathbb R^{n\times p}$, respectively, and invertible smooth matrix-valued functions $\beta^u(x)$, $\beta^v(x)$ and $\eta(x)$ with values in $GL(m,\mathbb R)$, $GL(s,\mathbb R)$ and $GL(p,\mathbb R)$, respectively, such that
	\begin{align}\label{Eq:fb-sys-equi}
		\begin{aligned}
			 \left[ {\begin{matrix}
					{\tilde f\circ \psi}&{\tilde g^{\tilde u}\circ \psi}&{\tilde g^{\tilde v}\circ \psi}\\
					{\tilde h\circ \psi}&{\tilde l^{\tilde u}\circ \psi}&0
			\end{matrix}} \right]=  \left[ {\begin{matrix}
					{\frac{{\partial \psi }}{{\partial x}}}&{\frac{{\partial \psi }}{{\partial x}}\gamma  }\\
					0&{\eta }
			\end{matrix}} \right]\left[ {\begin{matrix}
					{f}&{g^u}&{g^v}\\
					{h}&{l^u}&0
			\end{matrix}} \right]\left[ {\begin{matrix}
					1&0&0\\
					{\alpha^u}&{\beta^u}&0\\
					{\alpha^v  + \lambda \alpha^u}&{\lambda \beta^u}&{\beta^v }
			\end{matrix}} \right].
		\end{aligned}
	\end{align}
	The sys-fb-equivalence of two control systems will be denoted by $\Sigma^{uv}\mathop  \sim \limits^{sys-fb}\tilde \Sigma^{\tilde u \tilde v}$. If $\psi:U \rightarrow\tilde U$ is a local diffeomorphism between neighborhoods $U$ of a point $x_p$ and $\tilde U$ of a point $\tilde x_p=\psi(x_p)$, and $\alpha^u$, $\alpha^v$, $\lambda$, $\gamma$, $\beta^u$, $\beta^v$, $\eta$ are defined  on $U$, we will speak about local sys-fb-equivalence.
\end{defn}
The two ODECSs $\Sigma^{uv}$ and $\tilde \Sigma^{u\tilde v}$ of Proposition \ref{Pro:explinonDACS} are, by definition, system feedback equivalent with $\psi$ being identity, $\alpha^u=0$ and $\beta^u=I_m$. The following observation is crucial  and will play an important role for studying the feedback linearization problems of DACSs in Section~\ref{Sec:4}, which points out that the  feedback transformations  of explicitation systems  of DACSs have a   \emph{triangular form} which are different from  those of classical (ODE) control systems (see e.g., \cite{nijmeijer1990nonlinear,Isidori:1995:NCS:545735}).
\begin{rem}\label{rem:fb-sys-eq}
	Observe that, in  (\ref{Eq:fb-sys-equi}), there are two kinds of feedback transformations. Namely, $$u=\alpha^u(x)+\beta^u(x)\tilde u  \text{ \ \ and \ \ } v=\alpha^v(x)+\lambda(x)u+\beta^v(x)\tilde v,$$ which can be written together as a feedback transformation of $(u,v)$ with a (lower) \emph{triangular form}:
	\begin{align}\label{Eq:trianfb}
		\left[ {\begin{matrix}
				{u}\\
				{v}
		\end{matrix}} \right]=\left[ {\begin{matrix}
				{\alpha^u}(x)\\
				{\alpha^v}(x)
		\end{matrix}} \right]+\left[ {\begin{matrix}
				{\beta^u}(x)&0\\
				{\lambda}(x)&\beta^v(x)
		\end{matrix}} \right]\left[ {\begin{matrix}
				{\tilde u}\\
				{\tilde v}
		\end{matrix}} \right].
	\end{align}
	It implies that there are two kinds of inputs in the ODECSs of the form (\ref{Eq:controlsys1}), one input (the driving variable $v$) is more ``powerful'' than the other input (the original control variable $u$), since when transforming $v$, we can use both $u$ and $x$, but when transforming $u$, we are \emph{not} allowed to use $v$. Another difference between $u$ and $v$ is that the input $u$ is injected into the output $y$ via $l^uu$, but the driving variable $v$ is not directly injected into the output $y$. 

\end{rem}
The following theorem connects  ex-fb-equivalence of two DACSs with   sys-fb-equivalence of two ODECSs (explicitations).
\begin{thm} \label{Thm:mainnonDACS}
	Consider two DACSs $\Xi^u_{l,n,m}=(E,F,G)$ and $\tilde \Xi^{\tilde u}_{l,n,m}=(\tilde E,\tilde F,\tilde G)$ defined on $X$ and $\tilde X$, respectively.  Assume that ${\rm rank\,} {E(x)}=const.=r$ in a neighborhood $U$ of a point $x_p\in X$ and ${\rm rank\,}{\tilde E(\tilde x)}=r$ in a neighborhood $\tilde U$ of a point $\tilde x_p\in \tilde X$. Then, given any ODECSs $\Sigma^{uv}_{n,m,s,p}=(f,g^u,g^v,h,l^u)\in \mathbf{Expl}(\Xi^u)$ and $\tilde \Sigma^{\tilde u \tilde v}_{n,m,s,p}=(\tilde f,\tilde g^{\tilde u},\tilde g^{\tilde v},\tilde h,\tilde l^{\tilde u}) \in \mathbf{Expl}(\tilde \Xi^{\tilde u})$, we have that locally $\Xi^u\mathop  \sim \limits^{ex-fb}\tilde \Xi^{\tilde u}$ if and only if $\Sigma^{uv}\mathop  \sim \limits^{sys-fb}\tilde\Sigma^{\tilde u\tilde v}$.
\end{thm} 
\begin{proof}
	By the assumptions that  ${\rm rank\,} E(x)$ and ${\rm rank\,}\tilde E(x)$ are constant and equal to $r$ around $x_p$ and $\tilde x_p$, respectively,  there exist  invertible matrix-valued functions $Q:U\rightarrow GL(l,\mathbb R)$ and $\tilde Q:\tilde U\rightarrow GL(l,\mathbb R)$, defined on neighborhoods $U$ of $x_p$ and $\tilde U$ of $\tilde x_p$, respectively, such that 
	$E'(x)=Q(x)E(x)=\left[ {\begin{smallmatrix}
			E_1(x)\\
			0
	\end{smallmatrix}} \right]$ and   $\tilde E'(\tilde x)=\tilde Q(\tilde x)\tilde E(\tilde x)=\left[ {\begin{smallmatrix}
			\tilde E_1(\tilde x)\\
			0
	\end{smallmatrix}} \right]$, where $E_1:U\rightarrow R^{r\times n}$ and $\tilde E_1: \tilde U\rightarrow R^{r\times n}$  are of full row rank. We have $\Xi^u\mathop  \sim \limits^{ex-fb}\Xi^{u'}=(E',F',G')$ and $\tilde \Xi^{\tilde u}\mathop  \sim \limits^{ex-fb}\tilde \Xi^{\tilde u'}=(\tilde E',\tilde F',\tilde G')$ via $Q(x)$ and $\tilde Q(\tilde x)$, respectively, where $$
	\begin{array}{c}
		F'(x)=QF(x)=\left[ {\begin{matrix}
				F_1(x)\\
				F_2(x)
		\end{matrix}} \right],\quad G'(x)=QG(x)=\left[ {\begin{matrix}
				G_1(x)\\
				G_2(x)
		\end{matrix}} \right],\quad
		\tilde F'(\tilde x)=\tilde Q\tilde F(\tilde x)=\left[ {\begin{matrix}
				\tilde F_1(\tilde x)\\
				\tilde F_2(\tilde x)
		\end{matrix}} \right],\quad \tilde G'(\tilde x)=\tilde Q\tilde G(\tilde x)=\left[ \begin{matrix}
			\tilde G_1(\tilde x)\\
			\tilde G_2(\tilde x) \end{matrix} \right].
	\end{array}
	$$  In this proof, without loss of generality, we will assume that  $\Xi^u=\Xi^{u'}$ and $\tilde \Xi^{\tilde u}=\tilde \Xi^{\tilde u'}$, since  $\Xi^u\mathop  \sim \limits^{ex-fb}\tilde \Xi^{\tilde u}$ if and only if $ \Xi^{u'}\mathop  \sim \limits^{ex-fb}\tilde \Xi^{\tilde u'}$. 
	
	Moreover, choose   maps $f$, $g^u$, $g^v$, $h$, $l^u$ and $\tilde f$, $\tilde g^{\tilde u}$, $\tilde g^{\tilde v}$, $\tilde h$, $\tilde l^{\tilde u}$ such that
	\begin{align}\label{Eq:sys matrix}
		\begin{array}{c}
			f(x)=E^{\dagger}_1(x)F_1(x), \quad g^u(x)=E^{\dagger}_1(x)G_1(x), \quad 	{\rm Im\,}g^v(x)=\ker E_1(x), \quad h(x)=F_2(x), \quad	l^u(x)=G_2(x), \\
\tilde f(\tilde x)=\tilde E^{\dagger}_1(\tilde x)\tilde F_1(\tilde x),\quad  \tilde g^{\tilde u}(\tilde x)=\tilde E^{\dagger}_1(\tilde x)\tilde G_1(\tilde x),\quad {\rm Im\,}\tilde g^{\tilde v}(\tilde x)=\ker \tilde E_1(\tilde x), \quad \tilde h(\tilde x)=\tilde F_2(\tilde x),\quad \tilde l^{\tilde u}(\tilde x)=\tilde G_2(\tilde x),  
		\end{array}
	\end{align}
	where $E^{\dagger}_1(x)$ and $\tilde E^{\dagger}_1(\tilde x)$ are right inverses of $E_1(x)$ and $\tilde E_1(\tilde x)$, respectively. Then by Definition \ref{Def:explicitation},  
	$$
	\begin{aligned}
		\Sigma^{uv}=(f,g^u,g^v,h,l^u)\in \mathbf{Expl}(\Xi^u),\ \ \ \ \ \  \tilde \Sigma^{\tilde u \tilde v}=(\tilde f,\tilde g^{\tilde u},\tilde g^{\tilde v},\tilde h,\tilde l^{\tilde u})\in \mathbf{Expl}(\tilde \Xi^{\tilde u}).
	\end{aligned}
	$$ 
	It is seen from Proposition \ref{Pro:explinonDACS} that any control system in $\mathbf{Expl}(\Xi^u)$ is sys-fb-equivalent to $\Sigma^{uv}$ and that any control system in  $\mathbf {Expl}(\tilde \Xi^{\tilde u})$ is sys-fb-equivalent to $\tilde \Sigma^{\tilde u \tilde v}$. Without loss of generality, in the remaining part of the proof, we use $\Sigma^{uv}$ and $\tilde \Sigma^{\tilde u \tilde v}$ with system matrices given by (\ref{Eq:sys matrix}) to represent   two ODECSs in $\mathbf{Expl}(\Xi^u)$ and $\mathbf {Expl}(\tilde \Xi^{\tilde u})$, respectively. Throughout the proof below, we may drop the argument $x$ for the
	functions $E(x)$, $F(x)$, $G(x)$, ..., for ease of notation.
	
	\emph{If.} Suppose that locally $\Sigma^{uv}\mathop  \sim \limits^{sys-fb}\tilde \Sigma^{\tilde u \tilde v}$. Then there exist a  local diffeomorphism $\tilde x=\psi(x)$ and matrix-valued functions $\alpha^u$, $\alpha^v$, $\lambda$, $\gamma$, $\beta^u$, $\beta^v$, $\eta$  defined on a neighborhood $U$ of $x_p$ such that the system matrices satisfy relations (\ref{Eq:fb-sys-equi}) of Definition \ref{Def:sys-fb}.
	
	First, consider $\tilde g^{\tilde v}\circ\psi ={ \frac{{\partial \psi }}{{\partial x}}}g^v\beta^v $. By ${\rm Im\,}g^v=\ker E_1$, ${\rm Im\,}\tilde g^{\tilde v}=\ker \tilde E_1$, we have $\ker \tilde E_1\circ\psi={ \frac{{\partial \psi }}{{\partial x}} }\ker E_1$. Thus there exists $Q_1:U\to GL(r,\mathbb R)$ such that 
	\begin{align}\label{Eq:EtildeE}
		\tilde E_1\circ\psi=Q_1E_1\left(\frac{{\partial \psi }}{{\partial x}}\right)^{-1}.
	\end{align} 
	Then,  by (\ref{Eq:fb-sys-equi}), the following relation holds:
	\begin{align*}
		\left[ {\begin{matrix}
				{\tilde f\circ \psi}&{\tilde g^{\tilde u}\circ \psi}\\
				{\tilde h\circ \psi}&{\tilde l^{\tilde u}\circ \psi}
		\end{matrix}} \right] = \left[ {\begin{matrix}
				{\frac{{\partial \psi }}{{\partial x}}}&{\frac{{\partial \psi }}{{\partial x}}\gamma }\\
				0&{\eta }
		\end{matrix}} \right]\left[ {\begin{matrix}
				{f}&{g^u}&{g^v}\\
				{h}&{l^u}&0
		\end{matrix}} \right]\left[ {\begin{matrix}
				{{1}}&0\\
				{\alpha^u}&{\beta^u}\\
				{\alpha^v  + \lambda \alpha^u}&{\lambda \beta^u}
		\end{matrix}} \right].
	\end{align*}
	Substituting  (\ref{Eq:sys matrix}) into the above equation, we get
	$$
	\begin{aligned}
	 	\left[ {\begin{matrix}
				{\tilde E_1^{\dagger}\circ \psi \cdot \tilde F_1\circ \psi}&{\tilde E_1^{\dagger}\circ \psi\cdot\tilde  G_1\circ \psi}\\
				{\tilde F_2\circ \psi}&{\tilde G_2\circ \psi}
		\end{matrix}} \right] = \left[ {\begin{matrix}
				{\frac{{\partial \psi }}{{\partial x}}}&{\frac{{\partial \psi }}{{\partial x}}\gamma }\\
				0&{\eta }
		\end{matrix}} \right]\left[ {\begin{matrix}
				E^{\dagger}_1F_1&E^{\dagger}_1G_1&g^v\\
				{F_2}&{G_2}&0
		\end{matrix}} \right]\left[ {\begin{matrix}
				{{1}}&0\\
				{\alpha^u}&{\beta^u}\\
				{\alpha^v  + \lambda \alpha^u}&{\lambda \beta^u}
		\end{matrix}} \right].
	\end{aligned}
	$$
	Premultiply the above equation by $$
	\left[ {\begin{matrix}
			\tilde E_1\circ\psi&0\\
			0&I_p
	\end{matrix}} \right]=	\left[ {\begin{matrix}
			Q_1E_1\left(\frac{{\partial \psi }}{{\partial x}}\right)^{-1}&0\\
			0&I_p
	\end{matrix}} \right]$$ to  get
	\begin{align}\label{Eq:fgtildefg}
		\left[ {\begin{matrix}
				\tilde F_1\circ \psi&\tilde G_1\circ \psi\\
				{\tilde F_2\circ \psi}&{\tilde G_2\circ \psi}
		\end{matrix}} \right] = \left[ {\begin{matrix}
				Q_1&Q_1E_1\gamma  \\
				0&{\eta }
		\end{matrix}} \right]\left[ {\begin{matrix}
				F_1&G_1\\
				{F_2}&{G_2}
		\end{matrix}} \right]\left[ {\begin{matrix}
				1&0\\
				{\alpha^u}&{\beta^u}\\
		\end{matrix}} \right].
	\end{align}
	Now from equations (\ref{Eq:EtildeE}), (\ref{Eq:fgtildefg}) and Definition \ref{Def:ex-fb-equi}, it can be seen {that} $\Xi^{u}\mathop  \sim \limits^{ex-fb}\tilde \Xi^{\tilde u}$ via the transformations defined by $\tilde x=\psi(x)$, $Q=\left[ {\begin{smallmatrix}
			Q_1&Q_1E_1\gamma  \\
			0&{\eta }
	\end{smallmatrix}} \right]$, $\alpha^u$ and $\beta^u$.
	
	\emph{Only if.} Suppose that $\Xi^u\mathop  \sim \limits^{ex-fb}\tilde \Xi^{\tilde u}$ (in  a neighborhood $U$ of $x_p$).
	Assume that $\Xi^{u}$ and $\tilde \Xi^{u}$ are ex-fb-equivalent via an invertible matrix-valued function $Q=\left[ {\begin{smallmatrix}
			{{Q_1}}&{{Q_2}}\\
			{{Q_3}}&{{Q_4}}
	\end{smallmatrix}} \right]$,  $\tilde x =\psi(x)$, $\alpha^u$, $\beta^u$, where $Q_1:U\rightarrow\mathbb{R}^{r\times r}$ and $Q_2$, $Q_3$, $Q_4$ are matrix-valued functions of appropriate sizes.   Then by 
	$$
	Q E =\tilde E\circ\psi { \frac{{\partial \psi  }}{{\partial x}}}\Rightarrow\left[ {\begin{matrix}
			{{Q_1}}&{{Q_2}}\\
			Q_3&{{Q_4}}
	\end{matrix}} \right]\left[ {\begin{matrix}
			{{E_1}}\\
			0
	\end{matrix}} \right] = \left[ {\begin{matrix}
			\tilde E_1\circ \psi   \\
			0
	\end{matrix}} \right]\frac{{\partial \psi }}{{\partial x}},$$
	we can deduce that   
	\begin{align}\label{Eq:EtildeE1}
		\tilde E_1\circ \psi=Q_1E_1\left( \frac{{\partial \psi }}{{\partial x}} \right)^{-1}. 
	\end{align}
	Moreover, we have $Q_3=0$ and $Q_1$ is invertible (since both $E_1$ and $\tilde E_1$ are of full row rank), which implies {that} $Q_4$ is invertible as well (since $Q$ is invertible).
	Subsequently, by 
	\begin{align*}
		\begin{aligned}
			\tilde F\circ \psi =Q(F+G\alpha^u)\Rightarrow  \left[ {\begin{matrix}
					\tilde F_1\circ\psi\\
					\tilde F_2\circ\psi
			\end{matrix}} \right] =\left[ {\begin{matrix}
					{{Q_1}}&{{Q_2}}\\
					0&{{Q_4}}
			\end{matrix}} \right]\left(\left[ {\begin{matrix}
					F_1\\
					F_2
			\end{matrix}} \right]+\left[ {\begin{matrix}
					G_1\\
					G_2
			\end{matrix}} \right]\alpha^u \right),
		\end{aligned}
	\end{align*}
	we have 
	\begin{align}\label{Eq:f1}
		\tilde F_1\circ\psi  =Q_1(F_1+G_1\alpha^u)+Q_2(F_2+G_2\alpha^u) 
	\end{align}
	and 
	\begin{align} \label{Eq:f2}
		\tilde F_2\circ\psi &=Q_4(F_2+G_2\alpha^u).
	\end{align}
	Moreover, by $$\tilde G\circ\psi  =QG\beta^u\Rightarrow \left[ {\begin{matrix}
			\tilde G_1\circ\psi\\
			\tilde G_2\circ\psi
	\end{matrix}} \right]=\left[ {\begin{matrix}
			{{Q_1}}&{{Q_2}}\\
			0&{{Q_4}}
	\end{matrix}} \right]\left[ {\begin{matrix}
			G_1\\
			G_2
	\end{matrix}} \right]\beta^u,$$ we have
	\begin{align}\label{Eq:g1}
		\tilde G_1\circ \psi =Q_1G_1\beta^u+Q_2G_2\beta^u 
	\end{align}
	and 
	\begin{align}\label{Eq:g2}
		\tilde G_2\circ \psi=Q_4G_2\beta^u.
	\end{align}
	
	Recall the system matrices given in  (\ref{Eq:sys matrix}). First, from ${\rm Im\,}g^v=\ker E_1$, ${\rm Im\,}\tilde g^{\tilde v}\circ\psi=\ker \tilde E_1\circ\psi$, and equation (\ref{Eq:EtildeE1}), it is seen that there exists $\beta^v:U\rightarrow GL(s,\mathbb{R})$ such that
	\begin{align}\label{Eq:Bv}
		\tilde g^{\tilde v}\circ \psi={ \frac{{\partial \psi }}{{\partial x}} }g^v\beta^v.
	\end{align}
	Secondly, by equations (\ref{Eq:EtildeE1}) and (\ref{Eq:f1}), we have
	\begin{align}\label{Eq:proofDACS1} 
		\begin{aligned}
			\tilde f\circ \psi&=\tilde E^{\dagger}_1\circ \psi\tilde F_1\circ \psi  =\frac{{\partial \psi }}{{\partial x}} E^{\dagger}_1Q^{-1}_1 \left[ {\begin{matrix}
					Q_1&Q_2
			\end{matrix}} \right]\left[ {\begin{matrix}
					F_1+G_1\alpha^u\\F_2+G_2\alpha^u
			\end{matrix}} \right]  = \frac{{\partial \psi }}{{\partial x}} E^{\dagger}_1Q^{-1}_1\left[ {\begin{matrix}
					Q_1&Q_2
			\end{matrix}} \right]\left[ {\begin{matrix}
					F_1+G_1\alpha^u+E_1g^v\left(\lambda\alpha^u+\alpha^v \right)\\F_2+G_2\alpha^u
			\end{matrix}} \right]  \\
			&=\frac{{\partial \psi }}{{\partial x}} \left(f\!+\!g^u\alpha^u\!+\!g^v\left(\lambda\alpha^u\!+\!\alpha^v \right)\!+\!\gamma\left(h\!+\!l^u\alpha^u \right) \right), 
		\end{aligned}
	\end{align}
	where $\gamma=E^{\dagger}_1Q^{-1}_1Q_2$, and $\alpha^v$ and $\lambda$ are matrix-valued functions of appropriate sizes.  Thirdly, by equation (\ref{Eq:g1}), we have
	\begin{align}\label{Eq:proofDACS2}
		\begin{aligned}
			\tilde g^{\tilde u}\circ \psi&=\tilde E^{\dagger}_1\circ \psi\tilde G_1\circ \psi
			 =\frac{{\partial \psi }}{{\partial x}}E^{\dagger}_1Q^{-1}_1 \left[ {\begin{matrix}
					Q_1&Q_2
			\end{matrix}} \right]\left[ {\begin{matrix}
					G_1\beta^u\\G_2\beta^u
			\end{matrix}} \right]  
			 =\frac{{\partial \psi }}{{\partial x}} E^{\dagger}_1Q^{-1}_1\left[ {\begin{matrix}
					Q_1&Q_2
			\end{matrix}} \right]\left[ {\begin{matrix}
					G_1\beta^u+E_1g^v\lambda\\G_2\beta^u
			\end{matrix}} \right]  
		 =  \frac{{\partial \psi }}{{\partial x}} \left(g^u\beta^u+g^v\lambda+\gamma l^u\beta^u  \right).
		\end{aligned}
	\end{align}
	Note that we use the equations $E_1g^v\left(\lambda\alpha^u+\alpha^v \right)=0$ and $E_1g^v\lambda=0$ to deduce (\ref{Eq:proofDACS1}) and (\ref{Eq:proofDACS2}).  At last, by equations (\ref{Eq:f2}) and (\ref{Eq:g2}) we have 
	\begin{align}\label{Eq:proofDACS3}
		\tilde h\circ \psi =\tilde F_2\circ \psi=Q_4(F_2+G_2\alpha^u)=Q_4\left(h+l^u\alpha^u \right)
	\end{align}
	and 
	\begin{align}\label{Eq:proofDACS4}
		\tilde l^{\tilde u}\circ \psi=\tilde G_2\circ \psi=Q_4G_2\beta^u=Q_4l^u\beta^u .
	\end{align}
	Finally, it can be seen from (\ref{Eq:proofDACS1}), (\ref{Eq:proofDACS2}), (\ref{Eq:proofDACS3}) and (\ref{Eq:proofDACS4}), that $\Sigma^{uv}\mathop  \sim \limits^{sys-fb}\tilde \Sigma^{\tilde u \tilde v}$ via $\tilde x=\psi(x)$, $\alpha^v$, $\beta^v$, $\alpha^u$, $\beta^u$, $\lambda$, $\gamma=E^{\dagger}_1Q^{-1}_1Q_2$ and $\eta=Q_4$. 
\end{proof}
\section{External and internal feedback linearization}\label{Sec:4} 
In this section, we discuss the problem that when a nonlinear DACS of the form (\ref{Eq:DACS0}) is locally externally or internally feedback equivalent to a linear DACS of the form (\ref{Eq:linDAEcontol}) with complete controllability. 
First, we review some definitions and criteria for the complete controllability of linear DACSs. We denote by $A^{-1}\mathscr B$, the preimage of a space $\mathscr B$ under a linear map $A$. The augmented Wong sequences (see e.g.,  \cite{lewis1992tutorial,berger2013controllability,chen2021from}) of a linear DACS $\Delta^u_{l,n,m}=(E,H,L)$, given by  (\ref{Eq:linDAEcontol}), are  
\begin{align}
	\mathscr V_0:=\mathbb R^n, \quad\quad \mathscr V_{i+1}:=H^{-1}(E\mathscr V_{i}+{\rm Im\,}L), \ \ i\ge 0; \label{Eq:Vchap4}\\
	\mathscr W_0:= 0, \quad\quad \mathscr W_{i+1}:=E^{-1}(H\mathscr W_{i}+{\rm Im\,}L), \ \ i\ge 0. \label{Eq:Wchap4}
\end{align}
Additionally, recall the  following sequence of subspaces (see e.g. \cite{lewis1992tutorial}):
\begin{align}\label{Eq:Wchap4addition}
	\hat{\mathscr W}_1:= \ker E, \quad\quad  \hat{\mathscr W}_{i+1}:=E^{-1}(H\hat{\mathscr W}_{i}+{\rm Im\,}L), \  i\ge 1.
\end{align}	
For simplicity of notation, we denote $ 
	K_{\beta}={\rm diag}\{K_{\beta_1},\ldots,K_{\beta_k}\}\in \mathbb R^{(|\beta|-k)\times |\beta|}$, $
	L_{\beta} ={\rm diag}\{L_{\beta_1},\ldots,L_{\beta_k}\}\in \mathbb R^{(|\beta|-k)\times |\beta|}$, $
	\mathcal E_{\beta} ={\rm diag}\{e_{\beta_1},\ldots,e_{\beta_k}\}\in \mathbb R^{|\beta|\times k}$, $
	N_{\beta} ={\rm diag}\{N_{\beta_1},\ldots,N_{\beta_k}\}\in \mathbb R^{|\beta|\times |\beta|} $, 
where $\beta$ is a multi-index $\beta=(\beta_1,\dots,\beta_k)$ and $\left| \beta \right|=\sum\limits_{i=1}^{k} {\beta_i}$, and where
%$	K_{\beta_i}=\left[ {\begin{smallmatrix}
%		0&I_{{\beta_i}-1}
%\end{smallmatrix}} \right]\in \mathbb{R}^{({\beta_i}-1)  \times {\beta_i} }$, $e_{{\beta_i}}=\left[ {\begin{smallmatrix}
%0\\
%1
%\end{smallmatrix}} \right]\in \mathbb{R}^{{ {{\beta_i}}} \times 1 }$, $L_{\beta_i}=\left[ {\begin{smallmatrix}
%I_{{\beta_i}-1}&0
%\end{smallmatrix}} \right]\in \mathbb{R}^{({\beta_i}-1)  \times {\beta_i}}$, $N_{\beta_i}=\left[ {\begin{smallmatrix}
%0&0\\
%I_{{\beta_i}-1}&0
%\end{smallmatrix}} \right]\in \mathbb{R}^{{\beta_i}  \times {\beta_i}}$.
$$
\begin{array}{c}
	K_{\beta_i}=\left[ {\begin{smallmatrix}
			0&I_{{\beta_i}-1}
	\end{smallmatrix}} \right]\in \mathbb{R}^{({\beta_i}-1)  \times {\beta_i} }, \quad e_{{\beta_i}}=\left[ {\begin{smallmatrix}
			0\\
			1
	\end{smallmatrix}} \right]\in \mathbb{R}^{{ {{\beta_i}}} },\quad 	L_{\beta_i}=\left[ {\begin{smallmatrix}
			I_{{\beta_i}-1}&0
	\end{smallmatrix}} \right]\in \mathbb{R}^{({\beta_i}-1)  \times {\beta_i}}, \quad  N_{\beta_i}=\left[ {\begin{smallmatrix}
			0&0\\
			I_{{\beta_i}-1}&0
	\end{smallmatrix}} \right]\in \mathbb{R}^{{\beta_i}  \times {\beta_i}}. 
\end{array}
$$
Definition \ref{Def:ex-fb-equi}
applied to linear systems says that two linear DACSs
$\Delta^u_{l,n,m}\!=\!(E,H,L)$ and $\tilde \Delta^{\tilde u}_{l,n,m}\!=\!(\tilde E,\tilde H,\tilde L)$ are ex-fb-equivalent if
there exist constant invertible matrices $Q$, $P$, $S$ and a matrix $R$ such
that ${\tilde{E}=QEP^{-1}}$, $\tilde{H}=Q(H+LR)P^{-1}$, $\tilde{L}=QLS$.
\begin{defn} [complete controllability in \cite{berger2013controllability}]
	A linear DACS $\Delta^u_{l,n,m}=(E,H,L)$  is completely controllable if for any $x_0,x_1\in \mathbb R^n$,  there exist a solution $(x,u)$ of $\Delta^u$ and $t\in \mathbb R^{+}$ such that $x(0)=x_0$ and $x(t)=x_1$.
\end{defn}
\begin{lem}{\rm  \cite{berger2013controllability}}\label{Lem:controllability} 
	For a linear DACS $\Delta^u_{l,n,m}=(E,H,L)$, the following statements are equivalent: 
	\begin{enumerate}
		
		\item[(i)] $\Delta^u$ is completely controllable.

		\item[(ii)] ${\rm Im\,} E+{\rm Im\,}H+{\rm Im\,}L\!=\!{\rm Im\,} E+{\rm Im\,}L$ and ${\rm Im\,}_{\mathbb C }E+{\rm Im\,}_{\mathbb C }H+{\rm Im\,}_{\mathbb C }L={\rm Im\,}_{\mathbb C }(\lambda E-H)+{\rm Im\,}_{\mathbb C }L$, $ \forall\lambda\in \mathbb C$.
		
		\item[(iii)] $\mathscr V^*\cap \mathscr W^*=\mathbb R^n$, where $\mathscr V^*$ and $\mathscr W^*$ are the limits of the augmented Wong sequences  (\ref{Eq:Vchap4}) and (\ref{Eq:Wchap4}), respectively;
		
		\item[(iv)] $\Delta^u$ is ex-fb-equivalent (under linear transformations) to
		\begin{align*}
			\left[ {\begin{smallmatrix}
					{{I_{\left| \rho  \right|}}}&0\\
					0&{{L_{\bar \rho }}}\\0&0\\0&0
			\end{smallmatrix}} \right]\left[ {\begin{smallmatrix}
					{{{\dot \xi}_1}}\\
					{{{\dot \xi}_2}}
			\end{smallmatrix}} \right] = \left[ {\begin{smallmatrix}
					{N_\rho ^T}&0\\
					0&{{K_{\bar \rho }}}\\0&0\\0&0
			\end{smallmatrix}} \right]\left[ {\begin{smallmatrix}
					{{\xi_1}}\\
					{{\xi_2}}
			\end{smallmatrix}} \right] + \left[ {\begin{smallmatrix}
					{{\mathcal E_\rho }}&0\\
					0&0\\
					0&I_{m-m^*}\\0&0
			\end{smallmatrix}} \right] \left[ {\begin{smallmatrix}
					u_1\\
					u_2
			\end{smallmatrix}} \right],
		\end{align*}
		where $\rho=(\rho_1,\dots,\rho_{m^*})$ and $\bar \rho=(\bar \rho_1,\dots,\bar \rho_{s^*})$ are multi-indices, and  $s^*=n-\rk E$.  
	\end{enumerate}
\end{lem}
We define  (locally) internal and (locally) external feedback linearizability of nonlinear DACSs  as follows.
\begin{defn}  \label{Def:fb-lin-DACS}
	Consider a DACS $\Xi^u_{l,n,m}=(E,F,G)$ and fix an admissible point $x_a\in X$.
	Then $\Xi^u$ is called locally internally (resp. externally) feedback linearizable around $x_a$  if $\Xi^u$ is locally in-fb-equivalent (resp. ex-fb-equivalent) to a linear DACS with complete controllability around $x_a$.
\end{defn}
We consider an ODECS   $\Sigma^{uv}_{n,m,s,p}=(f,g^u,g^v,h,l^u)$, given by (\ref{Eq:controlsys1}). If  $\Sigma^{uv}$ has no outputs, we denote it by $\Sigma^{uv}_{n,m,s}=(f,g^u,g^v)$. Then for $\Sigma^{uv}_{n,m,s}=(f,g^u,g^v)$, define the following two sequences of  distributions $\mathcal D_i$ and $\hat{\mathcal D}_i$, called the \emph{linearizability distributions} of $\Sigma^{uv}$,
\begin{align}\label{Eq:disD1}
	 \left\lbrace {\begin{array}{l@{\ }l}
			\mathcal D_0&:=\{0\},\\
			\mathcal D_1&:={\rm span}\left\lbrace g^u_1,\ldots,g^u_m,g^v_1,\ldots,g^v_s \right\rbrace, \\
			\mathcal D_{i+1}&:=\mathcal D_i+[f,\mathcal D_i], \ \ \ i=1,2,\ldots,
	\end{array}} \right.\quad\quad \quad \left\lbrace \begin{array}{l@{\ }l}
		\\
		\hat{\mathcal D}_1&:={\rm span}\left\lbrace g^v_1,\ldots,g^v_s \right\rbrace, \\
		\hat{\mathcal D}_{i+1}&:=\mathcal D_{i}+[f,\hat{\mathcal D}_i], \ \ \ \ i=1,2,\ldots.
	\end{array}\right.  
\end{align}
\begin{rem}\label{rem:sequencesD}
	%	(i) The distribution sequences $\mathcal D_i$ and $\hat{\mathcal D}_i$ satisfy:
	%	$$
	%	\mathcal D_0\subseteq \hat{\mathcal D}_1\subseteq \mathcal D_1\subseteq \hat{\mathcal D}_2\subseteq \mathcal D_2 \dots\subseteq \hat{\mathcal D}_i\subseteq \mathcal D_i\subseteq\dots\subseteq \hat{\mathcal D}_{i^*},
	%	$$
	%	and either
	%	$
	%	\hat{\mathcal D}_{i^*}=\mathcal D_{i^*}=\hat{\mathcal D}_{i^*+j}=\mathcal D_{i^*+j}
	%	$
	%	or
	%	$
	%	\hat{\mathcal D}_{i^*}\subsetneq \mathcal D_{i^*}=\hat{\mathcal D}_{i^*+j}=\mathcal D_{i^*+j},
	%	$
	%	where $j\ge1$ and $i^*$ is the smallest $i$ such that $\mathcal D_{i}=\mathcal D_{i+1}$. Note that $i^*$ is not necessarily  the smallest $i$ such that $\hat{\mathcal D}_{i}=\hat{\mathcal D}_{i+1}$ (as seen in the second case, where $\hat{\mathcal D}_{i^*}\subsetneq\hat{\mathcal D}_{i^*+1}$). However, $\mathcal D_i$ and $\hat{\mathcal D}_i$ always have the same limit.
	Consider a linear DACS $\Delta^u=(E,H,L)$, denote $\mathscr W_i(\Delta^u)$ and $\hat {\mathscr W}_i(\Delta^u)$ as the subspaces $\mathscr W_i$, given by (\ref{Eq:Wchap4}), and $\hat {\mathscr W}_i$, given by (\ref{Eq:Wchap4addition}), of $\Delta^u$, respectively. For a linear ODECS  $\Lambda^{uv}=(A,B^u,B^v,C,D^u)$ (of the form (\ref{Eq:controlsys1}) but with constant system matrices), define the following two sequences of subspaces
	$$
	\begin{aligned}
		{\mathcal W}_0:=\{0\},\quad\quad  
		{{  {\mathcal{W}}}_{i + 1}} := \left[ {\begin{smallmatrix}
				A&B^w
		\end{smallmatrix}} \right]\left(  {\left[ {\begin{smallmatrix}
					{{{ {\mathcal{W}}}_i}}\\
					\mathbb R^{m+s}
			\end{smallmatrix}} \right] \cap \ker \left[ {\begin{smallmatrix}
					C&D^w
			\end{smallmatrix}} \right]} \right), \ i\ge0,
	\end{aligned}
	$$ 
	and	
	$$
	\begin{aligned}
		\hat {\mathcal W}_{1}:={{\rm Im\,} B^v}, \quad\quad
		{{\hat  {\mathcal{W}}}_{i + 1}} := \left[ {\begin{smallmatrix}
				A&B^w
		\end{smallmatrix}} \right]\left( {\left[ {\begin{smallmatrix}
					{{{\hat  {\mathcal{W}}}_i}}\\
				\mathbb R^{m+s}
			\end{smallmatrix}} \right] \cap \ker \left[ {\begin{smallmatrix}
					C&D^w
			\end{smallmatrix}} \right]} \right),\ i\ge1,
	\end{aligned}
	$$
	where $w=(u,v)$, $B^w=[B^u,B^v]$ and $D^w=[D^u,0]$. We have proved in Proposition 2.10 of \cite{chen2021from} that if $\Lambda^{uv}\in\mathbf{Expl}(\Delta^u)$, then 
	$$
	\begin{aligned}
		\mathscr W_i(\Delta^u)=\mathcal W_i(\Lambda^{uv}), \  \forall i\ge0, \quad \quad \hat {\mathscr W}_i(\Delta^u) =\hat {\mathcal W}_i(\Lambda^{uv}),\  \forall i\ge 1.
	\end{aligned}
	$$
	Apparently, ${\mathcal W}_i$ and $\hat {\mathcal W}_i$ are  linear counterparts of $\mathcal D_i$ and $\hat{\mathcal D}_i$, respectively, but they are for linear systems with outputs. 
\end{rem}
\begin{thm}[internal feedback linearization]\label{Thm:IFL}
	Consider a DACS $\Xi^u_{l,n,m}=(E,F,G)$, fix an admissible point $x_a\in X$. Let $M^*$ be the $n^*$-dimensional locally maximal controlled invariant submanifold of $\Xi^u$ around $x_a$.  Assume that the   constant rank assumption $\mathbf{(CR)}$   is satisfied  for $x\in M^*$ around $x_a$.  Then $\Xi^u|_{M^*}$ is a DACS $\Xi^{u^*}_{r^*,n^*,m^*}=(E^*,F^*,G^*)$ of the form (\ref{Eq:restr}) and its explicitation  $\mathbf{Expl}(\Xi^u|_{M^*})$   is a class of ODECSs without outputs. The DACS $\Xi^u$ is locally internally feedback linearizable  if and only if for one (and thus any) ODECS   $\Sigma^{u^*v^*}=(f^*,g^{u^*},g^{v^*})\in \mathbf{Expl}(\Xi^u|_{M^*})$, the linearizability distributions $\mathcal D_i $ and $\hat{\mathcal D}_i$ of $\Sigma^{u^*v^*}$ satisfy the following conditions on $M^*$ around $x_a:$
	\begin{enumerate} 
		\item [(FL1)] $\mathcal D_i$ and $\hat{\mathcal D}_i$ are  of constant rank for $1\le i \le n^*$.
		\item [(FL2)]  ${\mathcal D}_{n^*}=\hat{\mathcal D}_{n^*}=TM^*$.
		\item [(FL3)] $\mathcal D_i$ and $\hat{\mathcal D}_i$ are involutive for $1\le i \le n^*-1$.
	\end{enumerate}
\end{thm}
\begin{proof}
	Since $\Xi^u$ satisfies condition \textbf{(CR)} around $x_a$, its $M^*$-restriction $\Xi^u|_{M^*}$ by Definition \ref{Def:restrictionDACS} is a DACS $\Xi^u|_{M^*}=\Xi^{u^*}_{r^*,n^*,m^*}=(E^*,F^*,G^*)$ of  the  form (\ref{Eq:restr}) with $E^*$ being of full row rank $r^*$. It follows by the full row rankness of   $E^*$  that  the maps $h=F_2$ and $l^{u^*}=G_2$ are absent in the explicitation systems of $\Xi^{u^*}$, which means that the output $y=h(x)+l^{u^*}(x)u^*$ is absent as well (see Definition \ref{Def:explicitation}). Thus an ODECS  $\Sigma^{u^*v^*}_{n^*,m^*,s^*}=(f^*,g^{u^*},g^{v^*})\in~ \mathbf{Expl}(\Xi^u|_{M^*})$ is a control system without outputs, which is in the form 
	\begin{align*}
		\Sigma^{w^*}:\dot z^*= f^*(z^*)+g^{u^*}(z^*)u^*+g^{v^*}(z^*)v^*,
	\end{align*}
	where $w^*=(u^*,v^*)$,  $f^*=(E^*)^{\dagger}F^*$, $g^{u^*}=(E^*)^{\dagger}G^*$, $\im g^{v^*}=\ker E^*$ and   $s^*=n^*-r^*$.
	
	\emph{Only if.} Suppose that $\Xi^u$ is locally internally feedback linearizable, which means   that  its $M^*$-restriction $\Xi^u|_{M^*}$, given by (\ref{Eq:restr}), is locally ex-fb-equivalent %(via $Q(z^*)$, $\tilde z^*=\psi(z^*)$ and $u=\alpha(z^*)+\beta(z^*)\tilde u^*$ defined on $M^*$) 
	to a completely controllable linear DACS 
	\begin{align*}
		\Delta^{\tilde u^*}:E^*\dot {\tilde z}^*=H^*\tilde z^*+L^*\tilde u^*,
	\end{align*}
	where $E^*$, $H^*$, $L^*$ are constant matrices of appropriate sizes. Then a linear ODECS  $\Lambda^{\tilde w^*}=(A^*,B^{\tilde u^*},B^{\tilde v^*})\in \mathbf{Expl}(\Delta^{\tilde u^*})$, where $\tilde w^*=(\tilde u^*,\tilde v^*)$, is of the form
	\begin{align*}
		\Lambda^{\tilde w^*}:	\dot {\tilde z}^*= A^*{\tilde z}^*+B^{\tilde u^*}\tilde u^*+B^{\tilde v^*}\tilde v^*.
	\end{align*}
	where $A^*=(E^*)^{\dagger}H^*$, $B^{\tilde u^*}=(E^*)^{\dagger}L^*$ and $\im B^{\tilde v^*}=\ker E^*$.	By Lemma \ref{Lem:controllability}, the complete controllability of $\Delta^{\tilde u^*}$ implies $\hat{\mathscr W}_{n^*}(\Delta^{\tilde u^*})=\mathscr W_{n^*}(\Delta^{\tilde u^*})=\mathbb R^{n^*}$. By Proposition~2.10  of  \cite{chen2021from} ({see also} Remark \ref{rem:sequencesD}(ii)), we get $$\hat {\mathcal W}_{n^*}(\Lambda^{\tilde w^*})\!=\!\mathcal W_{n^*}(\Lambda^{\tilde w^*})\!=\!\hat{\mathscr W}_{n^*}(\Delta^{\tilde u^*})\!=\!\mathscr W_{n^*}(\Delta^{\tilde u^*})\!=\!\mathbb R^{n^*}.$$ Since $\Lambda^{\tilde w^*}$ is a linear control system without outputs, we have $\hat{\mathcal D}_{n^*}(\Lambda^{\tilde w^*})=\hat {\mathcal W}_{n^*}(\Lambda^{\tilde w^*})$, $\mathcal D_{n^*}(\Lambda^{\tilde w^*})=\mathcal W_{n^*}(\Lambda^{\tilde w^*})$. Hence,  $ \hat{\mathcal D}_{n^*}(\Lambda^{\tilde w^*})=\mathcal D_{n^*}(\Lambda^{\tilde w^*})=\mathbb R^{n^*}$. Thus $\Lambda^{\tilde w^*}$ satisfies  (FL2). Moreover, since $\Lambda^{\tilde w^*}$ is a linear control system, it satisfies (FL1) and (FL3) in an obvious way. Notice  that the nonlinear system $\Sigma^{w^*}$ is locally sys-fb-equivalent to $\Lambda^{\tilde w^*}$ by Theorem \ref{Thm:mainnonDACS} because $\Sigma^{w^*}\in \mathbf{Expl}(\Xi^u|_{M^*})$, $\Delta^{\tilde w^*}\in \mathbf{Expl}(\Delta^{\tilde u^*})$ and $\Xi^u|_{M^*}\overset{ex-fb}{\sim}\Delta^{\tilde u^*}$. Since $\Sigma^{w^*}$ and $\Lambda^{\tilde w^*}$ are control systems without outputs, sys-fb-equivalence reduces to feedback equivalence. Thus $\Sigma^{w^*}$ and $\Lambda^{\tilde w^*}$  are locally feedback equivalent (via $\tilde z^*=\psi(z^*)$ and two kinds of feedback transformations defined by   $\alpha^{u^*},\alpha^{v^*},\beta^{u^*},\beta^{v^*},\lambda$, see Remark \ref{rem:fb-sys-eq}). It is easy to verify by a direct calculation that if $ \hat{\mathcal D}_{i}$ and $\mathcal D_{i}$ are involutive, then the two distribution sequences are invariant for the two feedback equivalent  control systems $\Sigma^{w^*}$ and $\Lambda^{\tilde w^*}$, i.e., $ \frac{\partial \psi}{\partial z^*}\hat{\mathcal D}_{i}(\Sigma^{w^*})=\hat{\mathcal D}_{i}(\Delta^{\tilde w^*})\circ \psi$ and $ \frac{\partial \psi}{\partial z^*}{\mathcal D}_{i}(\Sigma^{w^*})={\mathcal D}_{i}(\Delta^{\tilde w^*})\circ \psi$.  So the system $\Sigma^{w^*}$ being feedback equivalent to $\Lambda^{\tilde w^*}$ satisfies conditions (FL1)-(FL3) as well. It is seen from Proposition \ref{Pro:explinonDACS} that any other ODECS $\hat \Sigma^{\hat w^*}\in \mathbf{Expl}(\Xi^u|_{M^*})$ is sys-fb-equivalent to $\Sigma^{w^*}$, which means $\Sigma^{w^*}$ is feedback equivalent (via two kinds of feedback transformations) to $\hat\Sigma^{\hat w^*}$ as any explicitation system in $\mathbf{Expl}(\Xi^u|_{M^*})$ has no outputs. So any other explicitation system $\hat\Sigma^{\hat w^*}$    satisfies (FL1)-(FL3) of Theorem~\ref{Thm:IFL} as well.
	
	\emph{If}. Suppose that an ODECS  $\Sigma^{u^*v^*}\in \mathbf{Expl}(\Xi^u|_{M^*})$ satisfies  (FL1)-(FL3) around $x_a$.   Then   the following lemma holds.
	\begin{lem}\label{cl:1}
		The ODECS  $$\Sigma^{w^*}=\Sigma^{u^*v^*}_{n^*,m^*,s^*}=(f^*,g^{u^*},g^{v^*})$$ is locally feedback equivalent, via  two kinds of feedback transformations (see Remark \ref{rem:fb-sys-eq}), to the Brunovsk{\'y} {canonical} form  \cite{brunovsky1970classification} around $x_a$, which is given by
		\begin{align}\label{Eq:bnskf}
			\Sigma^{\tilde w^*}_{Br}=\Sigma^{\tilde u^*\tilde v^*}_{Br}:
			\left\lbrace \begin{aligned}
				\dot \xi_1=N^T_{\rho}\xi_1+\mathcal E_{\rho}\tilde u^*,\\
				\dot \xi_2=N^T_{\bar \rho}\xi_2+\mathcal E_{\bar \rho}\tilde v^*,
			\end{aligned} \right.
		\end{align}
		where $\tilde w^*=(\tilde u^*,\tilde v^*)$, and $\rho=(\rho_1,\dots,\rho_a)$ and $\bar \rho=(\bar \rho_1,\dots,\bar \rho_b)$ are multi-indices.  
	\end{lem}
	The proof of Lemma \ref{cl:1}  is technical and is put into Appendix.  Now we {will} prove {that} the $M^*$-restriction $\Xi^u|_{M^*}$, given by (\ref{Eq:restr}), is locally ex-fb-equivalent to a linear DACS
	\begin{align}\label{Eq:DACScontrollable}
		\Delta^{\tilde u^*}:\left[ {\begin{matrix}
				{{I_{\left| \rho  \right|}}}&0\\
				0&{{L_{\bar \rho }}}
		\end{matrix}} \right]\left[ {\begin{matrix}
				{{{\dot \xi}_1}}\\
				{{{\dot \xi}_2}}
		\end{matrix}} \right] = \left[ {\begin{matrix}
				{N_\rho ^T}&0\\
				0&{{K_{\bar \rho }}}
		\end{matrix}} \right]\left[ {\begin{matrix}
				{{\xi_1}}\\
				{{\xi_2}}
		\end{matrix}} \right] + \left[ {\begin{matrix}
				{{\mathcal E_\rho }}\\
				0
		\end{matrix}} \right]\tilde u^*.
	\end{align}
	Notice that by Lemma \ref{Lem:controllability},   the   linear DACS $\Delta^{\tilde u^*}$  is completely controllable. 
	Observe that  $\Sigma^{\tilde w^*}_{Br}\in  \mathbf{Expl}(\Delta^{\tilde u^*})$, because the $\xi_1$-subsystems of $\Sigma^{\tilde w^*}_{Br}$ and $\Delta^{\tilde u^*}$ coincide,  $N^T_{\bar \rho}=L^{\dagger}_{\bar \rho }K_{\bar \rho }$ and $\ker L_{\bar \rho }=\im \mathcal E_{\bar \rho} $. Recall that $\Sigma^{w^*}$ is locally sys-fb-equivalent to $\Sigma^{\tilde w^*}_{Br}$  (by Lemma \ref{cl:1}) and  $\Sigma^{w^*}\in \mathbf{Expl}(\Xi^u|_{M^*})$, it is seen that $\Xi^u|_{M^*}$ is locally ex-fb-equivalent to  $\Delta^{\tilde u^*}$ around $x_a$    by Theorem \ref{Thm:mainnonDACS}. Hence $\Xi^u$ is locally in-fb-equivalent to the complete controllable linear DACS $\Delta^{\tilde u^*}$,  i.e., $\Xi^u$ is locally internally feedback linearizable. 
\end{proof}
\begin{thm}[external feedback linearization]\label{Thm:CEFL}
	Consider a DACS $\Xi^u_{l,n,m}=(E,F,G)$, fix an admissible  point $x_a\in X$. Then $\Xi^u$ is locally externally feedback linearizable  around $x_a$  if and only if there exists a neighborhood $U\subseteq X$ of $x_a$ in which the following conditions are satisfied.
	\begin{enumerate} 
		\item [(EFL1)]  ${\rm rank\,} E(x)$ and ${\rm rank\,}[E(x),G(x)]$ are constant.
		\item [(EFL2)] $F(x)\in {\rm Im\,}E(x)+{\rm Im\,} G(x)$ or, equivalently, the locally maximal invariant submanifold $M^*=M^c_0=U$.
		\item [(EFL3)]  For  one (and thus any)  control system $\Sigma^{uv}\in \mathbf{Expl}(\Xi^u|_{M^*})$,  which is a system with no outputs on $M^*=U$,  the  linerizability  distributions $\mathcal D_i$ and $\hat{\mathcal D}_i$ satisfy (FL1)-(FL3) of Theorem \ref{Thm:IFL}.
	\end{enumerate}
\end{thm}
\begin{proof}
	\emph{Only if.} Suppose that $\Xi^u$ is locally externally feedback linearizable.  By definition, the DACS $\Xi^u$ is locally ex-fb-equivalent  to a linear completely controllable DACS (via  $Q(x)$, $z=\psi(x)$ and $u=\alpha^u(x)+\beta^u(x) \tilde u$) 
	\begin{align}\label{Eq:lcDACS}
		\Delta^{\tilde u}:\tilde E\dot z=\tilde Hz+\tilde L \tilde u.
	\end{align} 
	Thus by Definition \ref{Def:ex-fb-equi}, we have
	\begin{align}\label{Eq:ex-fb-equi}
		\begin{array}{r@{\, }l}
			Q(x)E(x) = \tilde E\cdot{ \frac{{\partial \psi (x)}}{{\partial x}}},\quad Q(x)(F(x)+G(x)\alpha^u(x)) =\tilde H\cdot\psi (x),\quad  Q(x)G(x)\beta^u(x) =\tilde L.
		\end{array}
	\end{align}
	It is clear that $\Delta^{\tilde u}$  satisfies  (EFL1).  So the system $\Xi^u$ satisfies  (EFL1) as well because the {ranks} of $E(x)$ and $[E(x),G(x)]$ are invariant under ex-fb-equivalence. The complete controllability of $\Delta^{\tilde u}$ implies $\tilde Hz\in {\rm Im\,} \tilde E+{\rm Im\,} \tilde L$ (see Lemma \ref{Lem:controllability}(ii)). By substituting (\ref{Eq:ex-fb-equi}), we get
	$$
	\begin{aligned}
	 Q(F+G\alpha^u)(x)\in {\rm Im\,} QE\left( { \frac{{\partial \psi }}{{\partial x}}}\right) ^{-1}(x)+ {\rm Im\,}QG\beta^u(x) & \Rightarrow F(x)+G(x)\alpha^u(x)\in {\rm Im\,}E(x)+{\rm Im\,}G(x) \\& \Rightarrow F(x)\in  {\rm Im\,}E(x)+{\rm Im\,}G(x).
	\end{aligned}
	$$
	Thus $\Xi^u$ satisfies  (EFL2). Notice that by (EFL2), we have that the locally maximal controlled invariant submanifold $M^*$ around $x_a$ coincides with the neighborhood  $U$. Observe that the restriction $\Delta^{\tilde u}|_{M^*}=\Delta^{\tilde u}|_{U}$,  whose canonical form is given by
	\begin{align*}
		\left[ {\begin{matrix}
				{{I_{\left| \rho  \right|}}}&0\\
				0&{{L_{\bar \rho }}} 
		\end{matrix}} \right]\left[ {\begin{matrix}
				{{{\dot \xi}_1}}\\
				{{{\dot \xi}_2}}
		\end{matrix}} \right] = \left[ {\begin{matrix}
				{N_\rho ^T}&0\\
				0&{{K_{\bar \rho }}} 
		\end{matrix}} \right]\left[ {\begin{matrix}
				{{\xi_1}}\\
				{{\xi_2}}
		\end{matrix}} \right] + \left[ {\begin{matrix}
				{{\mathcal E_\rho }}\\
				0
		\end{matrix}} \right]
		u^*,
	\end{align*}
	is also a linear {completely} controllable DACS as $\Delta^{\tilde u}$. This means that $\Xi^u$ is locally internally feedback linearizable.  Thus by Theorem \ref{Thm:IFL}, the DACS $\Xi^u$ satisfies (EFL3) on $M^*=U$.
	
	\emph{If.} Suppose that in a neighborhood $U$ of $x_a$, the DACS $\Xi^u$ satisfies (EFL1)-(EFL3). Denote ${\rm rank\,} E(x)=r$, ${\rm rank\,} [E(x),G(x)]=r+\tilde m^*$ and $m^*=m-\tilde m^*$. Then, by (EFL1), there exist  an invertible $Q(x)$ defined on $U$ and a partition of $u=(u_1,u_2)$ such that
	\begin{align*}
		Q(x)E(x)\dot x=Q(x)F(x)+Q(x)G(x)u \Rightarrow \left[ {\begin{matrix}
				E_1(x)\\
				0\\
				0
		\end{matrix}} \right]\dot x=\left[ {\begin{matrix}
				F_1(x)\\
				F_2(x)\\
				F_3(x)
		\end{matrix}} \right]+\left[ {\begin{matrix}
				G^1_1(x)&G_1^2(x)\\
				G^1_2(x)&G^2_2(x)\\
				0&0
		\end{matrix}} \right]\left[ {\begin{matrix}
				u_1\\
				u_2
		\end{matrix}} \right],
	\end{align*}
	where $E_1(x)$ is of full row rank $r$ and $G^2_2(x)$ is a $\tilde m^*\times \tilde m^*$ invertible matrix-valued function defined on $U$. Moreover, by  (EFL2), we have $F_3(x)=0$ for $x\in U$. Now we use {the} feedback transformation
	\begin{align*}
		\left[ {\begin{matrix}
				\tilde u_1\\
				\tilde u_2
		\end{matrix}} \right]		= \left[ {\begin{matrix}
				0\\
				F_2(x)
		\end{matrix}} \right]+ \left[ {\begin{matrix}
				I_{m^*}&0\\
				G_2^1(x)&G_2^2(x)
		\end{matrix}} \right]\left[ {\begin{matrix}
				u_1\\
				u_2
		\end{matrix}} \right],
	\end{align*}
	and the system becomes
	\begin{align*}
		\left[ {\begin{matrix}
				E_1(x)\\
				0\\
				0
		\end{matrix}} \right]\dot x=\left[ {\begin{matrix}
				\tilde F_1(x)\\
				0\\
				0
		\end{matrix}} \right]+\left[ {\begin{matrix}
				\tilde 	G^1_1(x)&\tilde G_1^2(x)\\
				0&I_{\tilde m^*}\\
				0&0
		\end{matrix}} \right]\left[ {\begin{matrix}
				\tilde u_1\\
				\tilde u_2
		\end{matrix}} \right],
	\end{align*}
	where $\tilde F_1=F_1-G^2_1(G^2_2)^{-1}F_2$, $\tilde G^1_1=G^1_1-G^2_1(G^2_2)^{-1}G^1_2$ and $\tilde G^2_1=G^2_1(G^2_2)^{-1}$.
	
		Premultiply the above equation by $\left[ {\begin{smallmatrix}
			I_{r}&-\tilde G_1^2(x)&0\\
			0&I_{\tilde m^*}&0\\
			0&0&I_{l-r-\tilde m^*}
	\end{smallmatrix}} \right]$ to get
	\begin{align}\label{Eq:prf_Thm3_3}
		\left[ {\begin{matrix}
				E^*(x)\\
				0\\
				0
		\end{matrix}} \right]\dot x=\left[ {\begin{matrix}
				F^*(x)\\
				0\\
				0
		\end{matrix}} \right]+\left[ {\begin{matrix}
				G^*(x)&0\\
				0&I_{\tilde m^*}\\
				0&0
		\end{matrix}} \right]\left[ {\begin{matrix}
				u^*\\
				\tilde  u^*
		\end{matrix}} \right],
	\end{align}
	where $E^*=E_1$, $F^*=\tilde F_1$, $G^*=\tilde G_1^1$, $u^*=\tilde u_1$ and $\tilde u^*=\tilde u_2$. Then by Definition \ref{Def:restrictionDACS}, we have that $\Xi^u|_{M^*}=\Xi^u|_{U}$ is the following system:
	\begin{align*}
		\Xi^u|_{M^*}:	E^*(x)\dot x=	F^*(x)+	G^*(x)u^*.
	\end{align*}
	By  Theorem \ref{Thm:IFL} and condition (EFL3), $\Xi^u|_{M^*}$ is locally ex-fb-equivalent (on $M^*=U$) to a linear DACS $\Delta^{\tilde u^*}$ of the form  (\ref{Eq:DACScontrollable}). It follows from (\ref{Eq:prf_Thm3_3})  that $\Xi^u$ is locally on $U$ ex-fb-equivalent to
	\begin{align*}
		\left[ {\begin{smallmatrix}
				{{I_{\left| \rho  \right|}}}&0\\
				0&{{L_{\bar \rho }}}\\0&0\\0&0
		\end{smallmatrix}} \right]\left[ {\begin{smallmatrix}
				{{{\dot \xi}_1}}\\
				{{{\dot \xi}_2}}
		\end{smallmatrix}} \right] = \left[ {\begin{smallmatrix}
				{N_\rho ^T}&0\\
				0&{{K_{\bar \rho }}}\\0&0\\0&0
		\end{smallmatrix}} \right]\left[ {\begin{smallmatrix}
				{{\xi_1}}\\
				{{\xi_2}}
		\end{smallmatrix}} \right] + \left[ {\begin{smallmatrix}
				{{\mathcal E_\rho }}&0\\
				0&0\\
				0&I_{\tilde m^*}\\0&0
		\end{smallmatrix}} \right] \left[ {\begin{smallmatrix}
				u^*\\
				\tilde u^*
		\end{smallmatrix}} \right],
	\end{align*}
	which is {completely} controllable by Lemma \ref{Lem:controllability}. Therefore, $\Xi^u$ is locally externally feedback linearizable by Definition \ref{Def:fb-lin-DACS}.
\end{proof}
\begin{rem}\label{Rem:FL}
	(i) By conditions (EFL1) and (EFL2),  the locally maximal controlled invariant submanifold $M^*$ around $x_a$ is a neighborhood $U$ of $x_a$. So condition (EFL3) is actually, satisfied if and only if conditions (FL1)-(FL3) are satisfied on $M^*=U$, i.e., locally around $x_a$.

	(ii) Note that when applying the geometric reduction method of Definition \ref{Def:gr} to a linear DACS $\Delta^u=(E,H,L)$, we get a sequence of subspaces $\mathscr V_i=M_i$, which is actually  the augmented Wong sequence  $\mathscr V_i$ defined by (\ref{Eq:Vchap4}). Thus the locally maximal controlled invariant submanifold $M^*$ is a nonlinear generalization of the limit  $\mathscr V^*$ of $\mathscr V_i$.
	So condition (EFL2) together with condition $\hat{\mathcal D}_{n^*}=\mathcal D_{n^*}=TM^*$ of (FL2) are the nonlinear counterparts of condition $\mathscr V^*\cap \mathscr W^*=\mathbb R^n$ of Lemma~\ref{Lem:controllability}, which assures that the linearized DACS is completely controllable.   The  sequences of distributions $\mathcal D_i$ and $\hat{\mathcal D}_i$ can thus be seen as nonlinear generalizations of the augmented Wong sequence $\mathscr W_i$  of  (\ref{Eq:Wchap4}) and the sequence $\hat{\mathscr W}_i$ of (\ref{Eq:Wchap4addition}), respectively. 
	
	(iii) If $E(x)=I_n$, a DACS $\Xi^u=(E,F,G)$ becomes an ODECS of the form (\ref{Eq:ODEs}). Suppose that $G(x)=\left[ \begin{smallmatrix}
		g_1(x)&\ldots&g_m(x)
	\end{smallmatrix}\right] $ is of constant rank. We have that conditions   (EFL1)-(EFL2) of Theorem \ref{Thm:CEFL} are clearly satisfied and that condition (EFL3) reduces to the feedback linearizability conditions in the classical sense. Indeed, we have $\Xi^u\in\mathbf{Expl}(\Xi^u|_{M^*})=\mathbf{Expl}(\Xi^u)$ because $\Xi^u$ with $E(x)=I_n$ is already an ODECS. Thus the vector of driving variables $v$ is absent and the two linearizability distributions $\mathcal D_i$ and  $\hat{\mathcal D}_i$ satisfy $\hat{\mathcal D}_{i+1}=\mathcal D_i$ for $i\ge 1$. Hence conditions (FL1)-(FL3) become (FL1)' $\mathcal D_i$ are of constant rank for $1\le i\le n$; (FL2)' $\dim\mathcal D_n=n$; (FL3)'  $\mathcal D_i$ are involutive for $1\le i\le n-1$, which are the feedback linearizability conditions for  classical  nonlinear (ODE) control systems, see e.g., \cite{Jakubczyk&Respondek1980,Hunt&Su1981,Isidori:1995:NCS:545735,nijmeijer1990nonlinear}. 
\end{rem}
\section{Examples}\label{Sec:5} 
\begin{exa}
	Consider the following academic example borrowed from \cite{Berger2016zero}. For a DACS $\Xi^u$, defined on $X=\mathbb R^3$, given by
	\begin{align}\label{Eq:exa3}
		\left[ 	{\begin{matrix}
				x_2&x_1&0\\
				0&0&0\\
				1&0&1
		\end{matrix}}\right] \ \left[\begin{matrix}
			\dot x_1\\
			\dot x_2\\
			\dot x_3
		\end{matrix}\right]=\left[ 	{\begin{matrix}
				0\\
				0\\
				(x_2)^2-(x_1)^3+x_3
		\end{matrix}}\right]+\left[\begin{matrix}
			1&-1\\
			1&1\\
			0&0
		\end{matrix}\right] \ \left[\begin{matrix}
			u_1\\
			u_2
		\end{matrix}\right],
	\end{align}
	where $u=(u_1,u_2)$, we fix an admissible  point $x_a=(x_{1a}, x_{2a}, x_{3a})=(1,0,0)\in X.$  Clearly, there exists a neighborhood $U$  ($x_1\ne0$ for all $x\in U$) of $x_a$ such that conditions (EFL1) and (EFL2) of Theorem \ref{Thm:CEFL} are satisfied. Subsequently, via
	$
	Q=\left[ 	{\begin{smallmatrix}
			1&1&0\\
			0&0&1\\
			0&1&0
	\end{smallmatrix}}\right]$ and $ \left[\begin{smallmatrix}
		u_1\\
		u_2
	\end{smallmatrix}\right]=\left[\begin{smallmatrix} 
		1&0\\
		-1&1
	\end{smallmatrix}\right]\left[\begin{smallmatrix}
		\tilde u_1\\
		\tilde u_2
	\end{smallmatrix}\right]$,
	the DACS $\Xi^u$ is ex-fb-equivalent to
	\begin{align*}
		\left[ 	{\begin{matrix}
				x_2&x_1&0\\
				1&0&1\\
				0&0&0
		\end{matrix}}\right] \ \left[\begin{matrix}
			\dot x_1\\
			\dot x_2\\
			\dot x_3
		\end{matrix}\right]=\left[ 	{\begin{matrix}
				0\\
				(x_2)^2-(x_1)^3+x_3\\
				0
		\end{matrix}}\right]+\left[\begin{matrix}
			2&0\\
			0&0\\
			0&1
		\end{matrix}\right] \ \left[\begin{matrix}
			\tilde u_1\\
			\tilde u_2
		\end{matrix}\right].
	\end{align*}
	Observe that the locally maximal invariant submanifold $M^*=U$ and
	\begin{align*}
		\Xi^u|_{M^*}\!=\!\Xi^u|_{U}:\left[ 	{\begin{matrix}
				x_2&x_1&0\\
				1&0&1
		\end{matrix}}\right]   \left[\begin{matrix}
			\dot x_1\\
			\dot x_2\\
			\dot x_3
		\end{matrix}\right]\!=\!\left[ 	{\begin{matrix}
				0\\
				(x_2)^2-(x_1)^3+x_3
		\end{matrix}}\right]\!+\!\left[\begin{matrix}
			2\\
			0
		\end{matrix}\right] u^*,
	\end{align*}
	where $u^*=\tilde u_1$.	Now an ODECS  $\Sigma^{u^*v}\in \mathbf{Expl}(\Xi^u|_{M^*})$ {can be taken as}
	\begin{align*}
		\Sigma^{u^*v}:\left[\begin{matrix}
			\dot x_1\\
			\dot x_2\\
			\dot x_3
		\end{matrix}\right]=\left[ 	{\begin{matrix}
				0\\
				0\\
				(x_2)^2-(x_1)^3+x_3
		\end{matrix}}\right]+\left[ 	{\begin{matrix}
				0\\
				2/x_1\\
				0
		\end{matrix}}\right]u^*+\left[ 	{\begin{matrix}
				x_1\\
				-x_2\\
				-x_1
		\end{matrix}}\right]v,
	\end{align*}
	where $v$ is {a} driving variable. We calculate the distributions $\mathcal D_i$ and $\hat{\mathcal D}_i$ for {the} system $\Sigma^{u^*v}$ to get
	\begin{align*} 
		\hat{\mathcal D}_1={\rm span}\left\lbrace g^v\right\rbrace , \quad \mathcal D_1={\rm span}\left\lbrace g^{u^*},g^v\right\rbrace, \quad \mathcal D_2=\hat{\mathcal D}_2={\rm span}\left\lbrace g^{u^*},g^v,ad_f g^v\right\rbrace, 
	\end{align*}
	where $$
	g^v\!=\!\left[ 	{\begin{matrix}
			x_1\\
			-x_2\\
			-x_1
	\end{matrix}}\right], \quad g^{u^*}\!=\!\left[ 	{\begin{matrix}
			0\\
			2/x_1\\
			0
	\end{matrix}}\right], \quad ad_f g^v\!=\!\left[ 	{\begin{matrix}
			0\\
			0\\
			3(x_1)^3+2(x_2)^2+x_1
	\end{matrix}}\right].$$ 
	Clearly,  the distributions above are {of} constant {rank} and $\mathcal D_2=\hat{\mathcal D}_2=T_xU$ for all $x\in U$.  Additionally, $[g^{u^*},g^v]=0\in \mathcal D_1$ and $\hat{\mathcal D}_1$ is of rank one, {so} the distributions $\hat{\mathcal D}_1$, $\mathcal D_1$, $\hat{\mathcal D}_2$ are all involutive. Thus,   condition (EFL3) of Theorem \ref{Thm:CEFL} is satisfied. Therefore, system $\Xi^u$ is externally feedback linearizable.
	
	In fact, we can choose $\varphi^{u^*}(x)$ and $ \varphi^{v}(x)$ such that 
	\begin{align*}
		{\rm span}\left\lbrace d\varphi^v\right\rbrace =\mathcal D_1^{\bot}, \ \ \ \ {\rm span}\left\lbrace d\varphi^v, d\varphi^{u^*}\right\rbrace =\hat{\mathcal D}_1^{\bot}.
	\end{align*}
	Furthermore, use the  following coordinates change and feedback transformation  (note that the  feedback transformation below has a triangular form as we discussed  in Remark \ref{rem:fb-sys-eq})
	\begin{align*}
		\begin{array}{c}
			\xi =\varphi^{u^*}(x)=x_1x_2, ~~~ z_1=\varphi^{v}(x)=x_1+x_3,  ~~~z_2 =L_f\varphi^v(x)=-(x_1)^3+(x_2)^2+x_3, 
	\\ \left[ \begin{matrix}
				\tilde u^*\\
				\tilde v
			\end{matrix}\right] = \left[ \begin{matrix}
				2&0\\
				\frac{4x_2}{x_1}&-3(x_1)^3-x_1-2(x_2)^2
			\end{matrix}\right] \left[ \begin{matrix}
				u^*\\
				v
			\end{matrix}\right] +   \left[ \begin{matrix}
				0\\
				(x_2)^2-(x_1)^3+x_3
			\end{matrix}\right],
		\end{array}
	\end{align*}
	the system $\Sigma^{uv}$ becomes
	\begin{align*}
		\Lambda^{\tilde u^* \tilde v}: 
			\dot \xi=\tilde u^*, \ \
			\dot z_1=z_2,\ \
			\dot z_2=\tilde v.  
	\end{align*}
	Now by Theorem \ref{Thm:mainnonDACS},  $\Xi^u|_{M^*}$ is ex-fb-equivalent to the following linear DACS  
	\begin{align*}
		\Delta^{\tilde u^*}:\left[ 	{\begin{matrix}
				1&0&0\\
				0&1&0
		\end{matrix}}\right] \ \left[\begin{matrix}
			\dot \xi\\
			\dot z_1\\
			\dot z_2
		\end{matrix}\right]=\left[ 	{\begin{matrix}
				0&0&0\\
				0&0&1
		\end{matrix}}\right]\left[\begin{matrix}
			\xi\\
			z_1\\
			z_2
		\end{matrix}\right]+\left[\begin{matrix}
			1\\
			0
		\end{matrix}\right]\tilde u^*,
	\end{align*}
	since $\Sigma^{u^*v}\in \mathbf{Expl}(\Xi^u|_{M^*})$, $\Lambda^{\tilde u^* \tilde v} \in \mathbf{Expl}(\Delta^{\tilde u^*})$, and  $\Sigma^{u^*v}\mathop\sim\limits^{sys-fb}\Lambda^{\tilde u^* \tilde v} $.	Therefore,  the original DACS $\Xi^u$ is ex-fb-equivalent to the following completely controllable linear DACS:
	\begin{align*}
		\left[ 	{\begin{matrix}
				1&0&0\\
				0&1&0\\
				0&0&0
		\end{matrix}}\right] \ \left[\begin{matrix}
			\dot \xi\\
			\dot z_1\\
			\dot z_2
		\end{matrix}\right]=\left[ 	{\begin{matrix}
				0&0&0\\
				0&0&1\\
				0&0&0
		\end{matrix}}\right]\left[\begin{matrix}
			\xi\\
			z_1\\
			z_2
		\end{matrix}\right]+\left[\begin{matrix}
			1&0\\
			0&0\\
			0&1
		\end{matrix}\right] \ \left[\begin{matrix}
			\tilde u^*\\
			\tilde u_2
		\end{matrix}\right] 
	\end{align*}
	via $$Q=\left[ 	{\begin{matrix}
			1&1&0\\
			0&0&1\\
			0&1&0
	\end{matrix}}\right], \quad\quad
		\left[ \begin{matrix}
			\xi\\
			z_1\\
			z_2
		\end{matrix}\right]=\left[ \begin{matrix}
			x_1x_2\\
			x_1+x_3 \\
			-(x_1)^3+(x_2)^2+x_3
		\end{matrix}\right], \quad\quad \left[\begin{matrix}
			u_1\\
			u_2
		\end{matrix}\right]=\left[\begin{matrix} 
			1/2&0\\
			-1&1
		\end{matrix}\right]\left[\begin{matrix}
			\tilde u^*\\
			\tilde u_2
		\end{matrix}\right]. $$
\end{exa}
\begin{exa}\label{Ex:1}
	Consider the  model of a 3-link manipulator  \cite{arai1998nonholonomic} with active joints $1$ and   $2$, and a passive joint~3 (see Figure~\ref{Fig:1} below), the same model was used in \cite{chen2021nfs} to illustrate an applicable algorithm for the geometric reduction method, we will use it for the internal feedback linearization of DACSs in the present paper.  
	\begin{figure}[htp!]
		\newcommand{\nvar}[2]{%
			\newlength{#1}
			\setlength{#1}{#2}
		}
		
		% Define a few constants for drawing
		\nvar{\dg}{0.1cm}
		\def\dw{0.25}\def\dh{0.5}
		\nvar{\ddx}{1.5cm}
		
		% Define commands for links, joints and such
		\def\link{\draw [double distance=1.5mm, very thick] (0,0)--}
		
		\def\joint{%
			\filldraw [fill=white] (0,0) circle (5pt);
			\fill[black] circle (2pt);
		}
		\def\grip{%
			\draw[ultra thick](0cm,\dg)--(0cm,-\dg);
			%	\fill (0cm, 0.5\dg)+(0cm,1.5pt) -- +(0.6\dg,0cm) -- +(0pt,-1.5pt);
			%	\fill (0cm, -0.5\dg)+(0cm,1.5pt) -- +(0.6\dg,0cm) -- +(0pt,-1.5pt);
		}
		\def\robotbase{%
			\draw[rounded corners=8pt] (-\dw,-\dh)-- (-\dw, 0) --
			(0,\dh)--(\dw,0)--(\dw,-\dh);
			\draw (-0.5,-\dh)-- (0.5,-\dh);
			\fill[pattern=north east lines] (-0.5,-1) rectangle (0.5,-\dh);
		}
		
		% Draw an angle annotation
		% Input:
		%   #1 Angle
		%   #2 Label
		% Example:
		%   \angann{30}{$\theta_1$}
		\newcommand{\angann}[2]{%
			\begin{scope}[black]
				\draw [dashed, black] (0,0) -- (1.2\ddx,0pt);
				\draw [->, shorten >=3.5pt] (\ddx,0pt) arc (0:#1:\ddx);
				% Unfortunately automatic node placement on an arc is not supported yet.
				% We therefore have to compute an appropriate coordinate ourselves.
				\node at (#1/2-2:\ddx+8pt) {#2};
			\end{scope}
		}
		
		% Draw line annotation
		% Input:
		%   #1 Line offset (optional)
		%   #2 Line angle
		%   #3 Line length
		%   #5 Line label
		% Example:
		%   \lineann[1]{30}{2}{$L_1$}
		\newcommand{\lineann}[4][0.5]{%
			\begin{scope}[rotate=#2, black,inner sep=2pt]
				\draw[dashed, black!40] (0,0) -- +(0,#1)
				node [coordinate, near end] (a) {};
				\draw[dashed, black!40] (#3,0) -- +(0,#1)
				node [coordinate, near end] (b) {};
				\draw[|<->|] (a) -- node[fill=white] {#4} (b);
			\end{scope}
		}
		
		% Define the kinematic parameters of the three link manipulator.
		\def\thetaone{60}
		\def\Lone{2}
		\def\thetatwo{-90}
		\def\Ltwo{2}
		\def\thetathree{50}
		\def\Lthree{2}
		\begin{center}	
			\begin{tikzpicture}[scale=1]
				\robotbase
				%\angann{\thetaone}{$\theta_1$}
				%\lineann[0.7]{\thetaone}{\Lone}{$L_1$}
				\link(\thetaone:\Lone);
				\joint
				%	\draw [dashed, black] (0,0) -- node [left, yshift = 0.2cm]{$r$} (2.7,0.7);
				\draw  (-0.2,0)  node  [left] {joint $1$};
				\draw [<->,thick] (0,1) node (yaxis) [above] {$y$}
				|- (1,0) node (xaxis) [right] {$x$};
				\begin{scope}[shift=(\thetaone:\Lone), rotate=\thetaone]
					%\angann{\thetatwo}{$\theta_2$}
					%\lineann[-1.5]{\thetatwo}{\Ltwo}{$L_2$}
					\link(\thetatwo:\Ltwo);
					\joint
					\draw  (0,0.2)  node  [above] {joint $2$};
					\begin{scope}[shift=(\thetatwo:\Ltwo), rotate=-\thetaone]
						\angann{50}{$\theta$}
						\lineann[0.7]{\thetathree}{1}{$l$}
						\link(\thetathree:\Lthree);
						\joint
						\draw  (0,-0.2)  node  [below] {joint $3$:$(x,y)$}; 
						\begin{scope}[shift=(\thetathree:\Lthree), rotate=\thetathree]
							\grip
						\end{scope}
					\end{scope}
				\end{scope}
			\end{tikzpicture}
		\end{center}
		\caption{A 3-link manipulator with a free joint}\label{Fig:1}
	\end{figure}
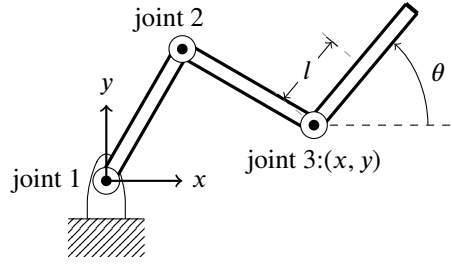 
	\\	The dynamic  equations of the manipulator are given by:
	\begin{align}\label{Eq:3-link} 
			\left\lbrace \begin{array}{r@{\,}l}
				m\ddot x-ml\sin\theta \ddot \theta-ml\dot \theta^2\cos \theta&=F_x,\\
				m\ddot y+ml\cos\theta \ddot \theta -ml\dot \theta^2\sin\theta&=F_y,\\
				-ml\sin \theta\ddot x+ml\cos \theta \ddot y+ml^2\ddot \theta&=\tau_\theta+F_f,
			\end{array}\right.  
	\end{align}
	where the mass $m$ and the half length of the free-link $l$ are constants, $x$ and $y$ are the position variables of the free joint, and $\theta$ is the angle between the base frame and the link frame, $F_x$ and $F_{y}$ are the translation force at the free joint in the direction of $x$ and $y$, respectively, and $\tau_{\theta}$ is the torque applied to the free joint (we take $\tau_{\theta}=0$ implying that joint 3 is free). We additionally consider the friction force  $F_f$  caused by the rotation of the free link.  We regard $(F_x,F_y)$ as the active control inputs  to the system. The friction force $F_f$ is  a generalized state variable rather than an active control input since we can not   change it arbitrarily.  
	%	 With the help of the following feedback transformations
	%		$$
	%	\left[ 	\begin{smallmatrix}
	%		u_1\\
	%		u_2
	%	\end{smallmatrix}\right] =	\left[ 	\begin{smallmatrix}
	%		\frac{ml\theta^2_2(\sin^2\theta_1-\cos^2\theta_1)}{d(\theta_1)}\\
	%		\frac{\theta^2_2(m\sin\theta_1-m\cos\theta_1)}{d(\theta_1)}
	%	\end{smallmatrix}\right] +	\left[ 	\begin{smallmatrix}
	%		\frac{-\cos\theta_1}{d(\theta_1)}&\frac{\sin\theta_1}{d(\theta_1)}\\
	%		\frac{m}{mld(\theta_1)}&\frac{m}{mld(\theta_1)}
	%	\end{smallmatrix}\right]\left[ 	\begin{smallmatrix}
	%		F_x\\
	%		F_y
	%	\end{smallmatrix}\right],  
	%	$$
	%	where $d(\theta_1)=m\sin\theta_1-m\cos\theta_1$, we transform system (\ref{Eq:3-link}) into  
	%	\begin{align}\label{Eq:3-link2}
	%		\begin{footnotesize}
	%			\left\lbrace \begin{array}{r@{\,}l}
	%				\ddot x&=u_1\\
	%				\ddot y&=u_2\\
	%				0&=-\sin \theta\ddot x+\cos \theta \ddot y+(l+\frac{I}{ml})\ddot \theta,
	%			\end{array}\right. 
	%		\end{footnotesize}
	%	\end{align}
	We consider system (\ref{Eq:3-link}) subjected to  the following  constraint:
	\begin{align}\label{Eq:holcon}
		x-y=0.
	\end{align}
	We combine (\ref{Eq:3-link}) together with (\ref{Eq:holcon}) as a DACS $\Xi_{7,7,2}^u=(E,F,G)$ of the form
	\begin{align*}  
		\left[ 	{\begin{smallmatrix}
				1& 0& 0& 0&0&0& 0\\
				0& m& 0&0& 0& -ml\sin\theta_1 &0\\
				0& 0& 1& 0&0&0& 0\\
				0& 0& 0& m&0&ml\cos\theta_1& 0\\
				0& 0& 0& 0&1&0 &0\\
				0&- \sin \theta_1& 0& \cos \theta_1&0&l& 0\\
				0& 0& 0& 0&0&0&  0   
		\end{smallmatrix}}\right] \left[\begin{smallmatrix}
			\dot x_1\\
			\dot x_2\\
			\dot y_1\\
			\dot y_2\\
			\dot \theta_1\\
			\dot \theta_2\\
			\dot F_f
		\end{smallmatrix}\right] =   \left[ 	{\begin{smallmatrix}
				x_2\\
				ml\theta^2_2\cos \theta_1\\
				y_2\\
				ml\theta^2_2\sin \theta_1\\
				\theta_2\\
				\frac{F_f}{ml}\\
				x_1-y_1
		\end{smallmatrix}}\right]+\left[\begin{smallmatrix}
			0&0\\
			1&0\\
			0&0\\
			0&1\\ 
			0&0\\
			0&0\\
			0&0
		\end{smallmatrix}\right] \left[\begin{smallmatrix}
			F_x\\
			F_y
		\end{smallmatrix}\right].   
	\end{align*}
	For the DACS $\Xi^u$,  the generalized states $\xi=(x_{1},x_{2},y_{1},y_{2},\theta_{1},\theta_{2},F_f)\in X= \mathbb R^6\times S$ and the vector of control inputs is $(F_x,F_y)$.
	Consider $\Xi^u$ around a point $\xi_{p}=(x_{1p},x_{2p},y_{1p},y_{2p},\theta_{1p},\theta_{2p},F_{fp})=0$.
	The system $\Xi^u$ is \emph{not} locally externally feedback linearizable since  condition (EF2) of Theorem \ref{Thm:CEFL} is not satisfied around $\xi_p$. Now we apply the geometric reduction method of Definition \ref{Def:gr} to   get 
	$$
	\begin{aligned}
		M^c_0= (-\frac{\pi}{2},\frac{\pi}{2})\times \mathbb R^6, \ \ M^c_1=\left\lbrace \xi \in M^c_0\,|\,  x_1-y_1=0\right\rbrace, \ \
		M^c_2&=\left\lbrace \xi \in M^c_1\,|\,  x_2-y_2=0\right\rbrace,~~~M^c_3=M^c_2.
	\end{aligned}
	$$
	Thus by Proposition \ref{Pro:M^*}, $M^*=M^c_3=M^c_2$ is the locally maximal  controlled invariant submanifold around $x_p\in M^*$ (so $x_p$ is admissible).
 Choose new coordinates $\xi_2=(\tilde x_1,\tilde x_2)= (x_1-y_1,x_2-y_2)$ and keep the remaining coordinates $\xi_1=(y_1,y_2,\theta_1,\theta_2,F_f)$ unchanged, the system represented in the new coordinates is
	\begin{align*} 
		\left[ 	{\begin{smallmatrix}
				1& 0&0&0& 0&1&0\\
				0&m& 0&-ml\sin\theta_1&0&0& m\\
				1& 0&0&0& 0&0& 0\\
				0& m&0&ml\cos\theta_1& 0&0& 0\\
				0& 0&1&0 &0&0& 0\\
				0& \cos \theta_1-\sin\theta_1&0&l& 0&	0&- \sin \theta_1\\
				0& 0&0&0&  0 &	0& 0
		\end{smallmatrix}}\right] \left[\begin{smallmatrix}
			\dot y_1\\
			\dot y_2\\
			\dot \theta_1\\
			\dot \theta_2\\
			\dot F_f\\
			\dot {\tilde x}_1\\
			\dot {\tilde x}_2\\
		\end{smallmatrix}\right]  = \left[ 	{\begin{smallmatrix}
				\tilde x_2+y_2\\
				ml\theta^2_2\cos \theta_1\\
				y_2\\
				ml\theta^2_2\sin\theta_1 \\
				\theta_2\\
				\frac{F_f}{ml}\\
				\tilde x_1
		\end{smallmatrix}}\right]+\left[\begin{smallmatrix}
			0&0\\
			1&0\\
			0&0\\
			0&1\\ 
			0&0\\
			0&0\\
			0&0
		\end{smallmatrix}\right] \left[\begin{smallmatrix}
			F_x\\
			F_y
		\end{smallmatrix}\right]. 
	\end{align*}
	Set  $\xi_2=(\tilde x_1,\tilde x_2)=0$ to get a DACS of the form (\ref{Eq:DACS})
	\begin{align*} 
		\left[ 	{\begin{smallmatrix}				
				1& 0&0&0& 0\\
				0&m& 0& -ml\sin\theta_1&0\\
				1& 0&0&0& 0\\
				0& m&0&ml\cos\theta_1& 0\\
				0& 0&1&0 &0\\
				0& \cos \theta_1-\sin\theta_1&0&l& 0\\  
		\end{smallmatrix}}\right] \left[\begin{smallmatrix}
			\dot y_1\\
			\dot y_2\\
			\dot \theta_1\\
			\dot \theta_2\\
			\dot F_f
		\end{smallmatrix}\right] = \left[ 	{\begin{smallmatrix}
				y_2\\
				ml\theta^2_2\cos \theta_1\\
				y_2\\
				ml\theta^2_2\sin \theta_1 \\
				\theta_2\\
				\frac{F_f}{ml}
		\end{smallmatrix}}\right]+\left[\begin{smallmatrix}
			0&0\\
			1&0\\
			0&0\\
			0&1\\ 
			0&0\\
			0&0
		\end{smallmatrix}\right] \left[\begin{smallmatrix}
			F_x\\
			F_y
		\end{smallmatrix}\right]. 
	\end{align*}
	By  using $Q(\xi_1)$ and the   feedback transformations defined on $M^*$ as
	$$
	\begin{aligned}
		Q(\xi_1)&=\left[ \begin{smallmatrix}
			1&0&0&0&0&0\\
			0&0&0&1&0&0\\
			0&0&0&0&1&0\\
			0&1&0&0&0&0\\
			0&\sin\theta_1&0&-\cos\theta_1&0&m \\
			1&0&-1&0&0&0\\
		\end{smallmatrix}\right], \ \ \ \  
		\left[ \begin{smallmatrix}
			u_1\\
			u_2
		\end{smallmatrix}\right] &= \left[ \begin{smallmatrix}
			0\\
			F_f/l
		\end{smallmatrix}\right]+ \left[ \begin{smallmatrix}
			1&0\\
			\sin\theta_1&-\cos\theta_1
		\end{smallmatrix}\right] \left[ \begin{smallmatrix}
			F_x\\
			F_y
		\end{smallmatrix}\right],
	\end{aligned}
	$$
	we bring the system into
	\begin{align*}
		\left[ 	{\begin{smallmatrix}
				1& 0&0&0& 0\\
				0& m&0&ml\cos\theta_1& 0\\
				0& 0&1&0 &0\\
				0&m& 0& -ml\sin\theta_1&0\\
				0& 0&0&0& 0\\ 
				0& 0&0&0& 0\\  
		\end{smallmatrix}}\right] \left[\begin{smallmatrix}
			\dot y_1\\
			\dot y_2\\
			\dot \theta_1\\
			\dot \theta_2\\
			\dot F_f
		\end{smallmatrix}\right] =  \left[ 	{\begin{smallmatrix}
				y_2\\
				ml\theta^2_2\sin \theta_1+\frac{F_f}{l}\sec\theta_1 \\
				\theta_2\\
				ml\theta^2_2\cos \theta_1\\
				0\\
				0
		\end{smallmatrix}}\right]+\left[\begin{smallmatrix}
			0&0\\
			\tan\theta_1&-\sec\theta_1\\
			0&0\\
			1&0\\ 
			0&1\\
			0&0
		\end{smallmatrix}\right] \left[\begin{smallmatrix}
			u_1\\
			u_2
		\end{smallmatrix}\right]. 
	\end{align*}
	The local $M^*$-restriction $\Xi^u|_{{M^*}}=(E^*,F^*,G^*)$ by Definition \ref{Def:restrictionDACS} (compare Example 5.1 of \cite{chen2021nfs}) is    
	\begin{align}\label{Eq:rsys}
		\left[ 	{\begin{smallmatrix}
				1& 0&0&0& 0\\
				0& m&0&ml\cos\theta_1& 0\\
				0& 0&1&0 &0\\
				0&m& 0& -ml\sin\theta_1&0 
		\end{smallmatrix}}\right] \left[\begin{smallmatrix}
			\dot y_1\\
			\dot y_2\\
			\dot \theta_1\\
			\dot \theta_2\\
			\dot F_f
		\end{smallmatrix}\right] =   \left[ 	{\begin{smallmatrix}
				y_2\\
				\frac{F_f}{l}\sec\theta_1+ml\theta^2_2\sin\theta_1 \\
				\theta_2\\
				ml\theta^2_2\cos \theta_1 
		\end{smallmatrix}}\right]+\left[\begin{smallmatrix}
			0\\
			\tan\theta_1\\ 
			0\\
			1
		\end{smallmatrix}\right]u_1.
	\end{align}
	An explicitation system $\Sigma^{u_1v}\in \mathbf{Expl}(\Xi^u|_{M^*})$ can be chosen as 
	\begin{align*} 
		\left[\begin{smallmatrix}
			\dot y_1\\
			\dot y_2\\
			\dot \theta_1\\
			\dot \theta_2\\
			\dot F_f
		\end{smallmatrix}\right] =\left[ 	{\begin{smallmatrix}
				y_2\\
				\frac{F_f\tan\theta_1+ml^2\theta_2^2}{ml(\cos\theta_1+\sin \theta_1)} \\
				\theta_2\\  \frac{F_f\sec\theta_1+ml^2\theta_2^2(\sin \theta_1-\cos\theta_1)}{ml^2(\cos\theta_1+\sin \theta_1)}\\
				0 
		\end{smallmatrix}}\right]+ \left[\begin{smallmatrix}
			0\\
			\frac{\sec\theta_1}{m(\cos\theta_1+\sin \theta_1)}\\
			0\\ 
			\frac{\tan\theta_1-1}{ml(\cos\theta_1+\sin \theta_1)}\\
			0
		\end{smallmatrix}\right]  
		u_1+\left[\begin{smallmatrix}
			0\\
			0\\
			0\\ 
			0\\
			1
		\end{smallmatrix}\right]  
		v. 
	\end{align*} 
	Define a new control 
	$$
	u^*:=	\begin{matrix}\frac{F_f\tan\theta_1+ml^2\theta_2^2}{ml(\cos\theta_1+\sin \theta_1)}+	\frac{\sec\theta_1}{m(\cos\theta_1+\sin \theta_1)}\end{matrix}u_1.
	$$
	Then the system $\Sigma^{u_1v}$ under the new control is 	$\Sigma^{u^*v}=(f,g^{u^*},g^{v})$:
	\begin{align*}
		\begin{array}{c}
			\left[\begin{smallmatrix}
				\dot y_1\\
				\dot y_2\\
				\dot \theta_1\\
				\dot \theta_2\\
				\dot F_f
			\end{smallmatrix}\right] \!=\!\left[ 	{\begin{smallmatrix}
					y_2\\
					0 \\
					\theta_2\\  
					\frac{F_f}{ml^2}\\
					0 
			\end{smallmatrix}}\right]+\left[\begin{smallmatrix}
				0\\
				1\\
				0\\ 
				\frac{1}{l}(\sin\theta_1-\cos\theta_1)\\
				0
			\end{smallmatrix}\right]  
			u^*+\left[\begin{smallmatrix}
				0\\
				0\\
				0\\ 
				0\\
				1
			\end{smallmatrix}\right]  
			v.
		\end{array}
	\end{align*} 
	Now calculate the distributions $\mathcal D_i$ and $\hat{\mathcal D}_i$ for {the} system $\Sigma^{u^*v}$ to get
	$ 
		\hat{\mathcal D}_1={\rm span}\left\lbrace g^v\right\rbrace$,  $
		\mathcal D_1={\rm span}\left\lbrace g^{u^*},g^v\right\rbrace$, 
		$\hat{\mathcal D}_2={\rm span}\left\lbrace g^{u^*},g^v,ad_f g^v\right\rbrace$, $\mathcal D_2={\rm span}\left\lbrace g^v,g^{u^*},ad_f g^v,ad_f g^{u^*}\right\rbrace$, $\mathcal D_3=\hat{\mathcal D}_2=TM^*$. 
	where   $$ 
	\begin{aligned}
g^v=\frac{\partial}{\partial F_f}, \quad ad_f g^v= -\frac{1}{ml^2}\frac{\partial}{\partial \theta_2}, \quad g^{u^*}=\frac{\partial}{\partial y_2}+\frac{1}{l}(\sin\theta_1-\cos\theta_1)\frac{\partial}{\partial \theta_2},\\ ad_f g^{u^*}=  -\frac{\partial}{\partial y_1}- \frac{1}{l}(\sin\theta_1-\cos\theta_1)\frac{\partial}{\partial \theta_1}+\frac{1}{l}(\sin\theta_1+\cos\theta_1)\theta_2\frac{\partial}{\partial \theta_2}.
	\end{aligned}
$$ 
	Clearly,  the distributions above are {of} constant {rank} and  are all involutive around $\xi_p$. Thus,  conditions (FL1)-(FL3) of Theorem~\ref{Thm:IFL} are satisfied. Therefore, system $\Xi^{u}$ is locally internally feedback linearizable around $\xi_p$. Indeed,  choose $\varphi^{u^*}(x)$ and $ \varphi^v(x)$ such that 
	\begin{align*}
		{\rm span}\left\lbrace d\varphi^v\right\rbrace =\mathcal D_2^{\bot}, \ \ \ \ {\rm span}\left\lbrace d\varphi^v, d\varphi^{u^*}\right\rbrace =\hat{\mathcal D}_2^{\bot}.
	\end{align*}
	Then define the  following coordinates change and feedback transformation (which has a triangular form as desired):  
	$$
	\begin{array}{c}
		\tilde y_1=\varphi^v(\xi_1)=y_1-l\int a(\theta_1)\rd \theta_1, \quad 
		\tilde y_2=L_f\varphi^v(\xi_1)=y_2- la(\theta_1)\theta_2, 
		\quad  \tilde F_f=L^2_f\varphi^v(\xi_1)=-a(\theta_1)F_f-a'(\theta_1)l\theta^2_2,\\
		\tilde \theta_1=\varphi^{u^*}(\xi_1) =\theta_1, \quad
		\tilde \theta_2=L_f\varphi^{u^*}(\xi_1)=\theta_2, \quad  
		\left[ \begin{smallmatrix}
			\tilde u^*\\
			\tilde v
		\end{smallmatrix}\right]= \left[ \begin{smallmatrix}
			\frac{1}{l}(\sin\theta_1-\cos\theta_1) &0\\
			-2a'(\theta_1)(\sin\theta_1-\cos\theta_1)\theta_2	&-a(\theta_1)
		\end{smallmatrix}\right]\left[ \begin{smallmatrix}
			u^*\\
			v
		\end{smallmatrix}\right]   +\left[ \begin{smallmatrix}
			\frac{F_f}{ml^2}\\
			-3a'(\theta_1)\theta_2F_f-a''(\theta_1)\theta^3_2l
		\end{smallmatrix}\right], 
	\end{array}
$$
	where $a(\theta_1)=\frac{1}{\sin\theta_1-\cos\theta_1}$, $a'(\theta_1)=\frac{\rd a(\theta_1)}{\rd \theta_1}$, $a''(\theta_1)=\frac{\rd^2 a(\theta_1)}{\rd \theta_1^2}$. We transform   $\Sigma^{u^*v}$  into a linear control system in the Brunovsk{\'y} form
	$$
		\Lambda^{\tilde u^*\tilde v}:  \left\lbrace \begin{matrix}
			\dot {\tilde y}_1&=&\tilde y_2,\\
			\dot {\tilde y}_2&=&\tilde F_f,\\
			\dot {\tilde F}_f&=&\tilde v,\\
			\dot {\tilde \theta}_1&=&\tilde \theta_2,\\
			\dot {\tilde \theta}_2&=&\tilde u^*.
		\end{matrix}  \right. 
$$
	Thus by Theorem \ref{Thm:mainnonDACS}, the restriction $\Xi^u|_{M^*}$, given by (\ref{Eq:rsys}), is locally ex-fb-equivalent to the following completely controllable linear DACS $\Delta^{\tilde u^*}$,
	\begin{align*}
		\Delta^{\tilde u^*}:\left[ 	{\begin{matrix}
				1&0&0&0&0\\
				0&1&0&0&0\\
				0&0&0&1&0\\
				0&0&0&0&1\\
		\end{matrix}}\right]   \left[\begin{matrix}
			\dot {\tilde y}_1\\
			\dot {\tilde y}_2\\
			\dot {\tilde F}_f\\
			\dot {\tilde \theta}_1\\
			\dot {\tilde \theta}_2
		\end{matrix}\right]=\left[ 	{\begin{matrix}
				0&1&0&0&0\\
				0&0&1&0&0\\
				0&0&0&0&1\\
				0&0&0&0&0\\
		\end{matrix}}\right]\left[\begin{matrix}
			{\tilde y}_1\\
			{\tilde y}_2\\
			{\tilde F}_f\\
			{\tilde \theta}_1\\
			{\tilde \theta}_2
		\end{matrix}\right]+\left[\begin{matrix}
			0\\
			0\\
			0\\
			0\\
			1
		\end{matrix}\right]\tilde u^*.
	\end{align*}
	because $\Sigma^{u^*v}\overset{sys-fb}{\sim}\Sigma^{u_1v}\in \mathbf{Expl}(\Xi^u|_{M^*})$,  $\Lambda^{\tilde u^* \tilde v} \in \mathbf{Expl}(\Delta^{\tilde u})$  and  $\Sigma^{u^*v}\mathop\sim\limits^{sys-fb}\Lambda^{\tilde u^* \tilde v} $.
	Hence the  DACS $\Xi^u$ is locally in-fb-equivalent to the  linear  DACS $\Delta^{\tilde u^*}$, i.e., $\Xi^u$ is locally internally feedback linearizable.
\end{exa}
\section{Conclusions and perspectives}\label{Sec:6} 
In this paper, we give necessary and sufficient conditions for the problem that when a nonlinear DACS is locally internally  or locally externally  feedback equivalent to a completely controllable linear DACS. The conditions are based on an ODECS constructed by
the explicitation with driving variables. Two examples are given to illustrate   how to externally or internally feedback linearize a nonlinear DACS.

A natural problem for future works is that of when a nonlinear DAE system is ex-fb-equivalent to a linear one which is not necessarily completely controllable. Actually, this problem is more involved than the problem of   external feedback linearization with complete controllability. Indeed, since in Theorem  \ref{Thm:CEFL}, the maximal controlled invariant submanifold $M^*$ on $U$ is $M^*=U$, it follows that the algebraic constraints are directly governed by some variables of $u$. Thus the in-fb-equivalence is very close to the ex-fb-equivalence. However, if $M^*\ne U$, then the algebraic constraints may affect the generalized state. Moreover, since the explicitation is defined up to a generalized output injection, it may happen that one system of the explicitation is feedback linearizable but another is not. The general feedback linearizability problem remains open and, in view of the above points, is challenging.

\section*{Appendix} 
\begin{proof}[Proof of Lemma \ref{cl:1}]
	For ease of notation, we drop the index ``$*$'' for $z^*$,  ${u^*}$, $v^*$ and $f^*$ of the system $\Sigma^{u^*v^*}_{n^*,m^*,s^*}$, that is, $\Sigma^{u^*v^*}$ becomes
	$$
	\Sigma^{uv}: \dot z=f(z)+g^u(z)u+g^v(z)v.
	$$	
	The admissible point $x_a$ in the $z$-coordinates will be denoted by $z_a$. We will only show the proof for the case that $$m^*=s^*=1,\ \ \ \ \ \ \rk [g^v(z_a)\ \ g^u(z_a)]=2.$$   The proof for the  general case (i.e., for any $m^*\ge1$ and $s^*\ge1$,  and for $\rk [g^v(z_a)\ \ g^u(z_a)]=m^*+s^*$) can be done in a similar fashion as that on page 233-238 of \cite{Isidori:1995:NCS:545735}  for the feedback linearization of nonlinear multi-inputs multi-outputs control systems. We now describe a procedure to construct a change of coordinates $\xi=\psi(z)$ and    a feedback transformation:
	\begin{align}\label{Eq:twofbDACS}
		\left[ {\begin{matrix}
				u\\
				v
		\end{matrix}} \right] = \left[ {\begin{matrix}
				{{\alpha^u}(z)}\\
				{{\alpha^v}(z)}
		\end{matrix}} \right] + \left[ {\begin{matrix}
				{{\beta^u}(z)}&0\\
				{\lambda(z)}&{{\beta^v}(z)}
		\end{matrix}} \right]\left[ {\begin{matrix}
				{\tilde u}\\
				{\tilde v}
		\end{matrix}} \right]
	\end{align}
	to transform $\Sigma^{uv}$ into its Brunovsk{\'y} {canonical} form, where $\beta^u,\beta^v, \alpha^u, \lambda,\alpha^v$ are scalar functions, and $\beta^u(z)$ and $\beta^v(z)$ are nonzero around $z_a$, notice that the designed feedback transformation (\ref{Eq:twofbDACS}) has a triangular form as in (\ref{Eq:trianfb}). Note that constructing (\ref{Eq:twofbDACS}) is equivalent to finding the inverse feedback transformation 
	\begin{align}\label{Eq:twofbDACS2}
		\left[ {\begin{matrix}
				\tilde u\\
				\tilde v
		\end{matrix}} \right] = \left[ {\begin{matrix}
				{{a^u}(z)}\\
				{{a^v}(z)}
		\end{matrix}} \right] + \left[ {\begin{matrix}
				{{b^u}(z)}&0\\
				{\tilde \lambda(z)}&{{b^v}(z)}
		\end{matrix}} \right]\left[ {\begin{matrix}
				{u}\\
				{v}
		\end{matrix}} \right].
	\end{align}
	where $$
	\begin{aligned}
		a^u=-(\beta^u)^{-1}\alpha^u, \quad a^v=(\beta^v)^{-1}\lambda(\beta^u)^{-1}\alpha^u-(\beta^v)^{-1}\alpha^v\quad
		b^u=(\beta^u)^{-1},\quad b^v=(\beta^u)^{-1}, \quad \tilde \lambda=-(\beta^v)^{-1}\lambda(\beta^u)^{-1}.
	\end{aligned}
	$$	
	Below we will search for functions $a^u$, $a^v$, $\tilde \lambda$, and nonzero functions $b^u$, $b^v$ to construct   (\ref{Eq:twofbDACS2}).

	Consider the two sequences of distributions $\mathcal D_i$ and $\hat{\mathcal D}_i$ for $\Sigma^{uv}$, given by (\ref{Eq:disD1}),  and define 
	$$
	\begin{aligned}
		\rho :=\max \left\lbrace i\in \mathbb N^+\,|\,  \hat{\mathcal D}_{i}\ne\mathcal D_{ i}\right\rbrace, \quad\quad \bar \rho :=\max\left\lbrace i\in \mathbb N^+\,|\,   \mathcal D_{ i-1}\ne \hat{\mathcal D}_{ i} \right\rbrace.
	\end{aligned}
	$$
	By $m^*=s^*=1$, it is seen that, for each $i\ge 1$,
	\begin{align}\label{Eq:dim}
		\begin{aligned}
			\dim\,  \mathcal D_{ i}-\dim \hat{\mathcal D}_{i}&=\left\lbrace\begin{matrix}
				0, \text{          if } \mathcal D_{ i}= \hat{\mathcal D}_{i}\\
				1, \text{          if } \mathcal D_{ i}\ne\hat{\mathcal D}_{i}
			\end{matrix} \right., \quad\quad \dim\, \hat{\mathcal D}_{ i}-\dim\, \mathcal D_{ i-1} &=\left\lbrace\begin{matrix}
				0, \text{          if } \hat{\mathcal D}_{ i}= \mathcal D_{ i-1}\\
				1, \text{          if }  \hat{\mathcal D}_{ i}\ne \mathcal D_{ i-1}
			\end{matrix} \right..
		\end{aligned}
	\end{align}
	It follows that $\rho+\bar \rho=n^*$. Then only two cases are possible:  either $\rho\ge \bar \rho$ or $\rho< \bar \rho$.

	Case 1: If $\rho\ge \bar \rho$, then we have 
	\begin{align*}
		\mathcal D_0\subsetneq \hat{\mathcal D}_{1}\subsetneq  \dots\subsetneq{\mathcal D}_{\bar \rho-1}\subsetneq  \hat{\mathcal D}_{\bar \rho}\subsetneq \mathcal D_{\bar \rho}=\hat{\mathcal D}_{\bar \rho+1}\subsetneq {\mathcal D}_{\bar \rho+1}=  \dots \subsetneq {\mathcal D}_{\rho-1}= \hat{\mathcal D}_{\rho}\subsetneq {\mathcal D}_{\rho}=\hat{\mathcal D}_{\rho+j}= {\mathcal D}_{\rho+j}, \ j>0.
	\end{align*} 
	It follows that ${\mathcal D}_{\rho}= {\mathcal D}_{n^*}=\hat{\mathcal D}_{n^*}$ Then by (FL2) of Theorem \ref{Thm:IFL}, we have ${\mathcal D}_{\rho}=TM^*$ and thus $\dim {\mathcal D}_{\rho}=n^*$. 
	%by a dimensional argument, we can deduce  $\rho+\bar \rho=n^*$. 
	By $\hat{\mathcal D}_{\rho}\subsetneq {\mathcal D}_{\rho}$ and  (\ref{Eq:dim}), we have $\dim \hat{\mathcal D}_{\rho}=n^*-1$.
	Now by the involutivity of $\hat{\mathcal D}_{\rho}$   (condition (FL3)),  we can choose a scalar function  $h^u(z)$ such that 
	\begin{align*}
		{\rm span}\left\lbrace dh^u\right\rbrace=\hat{\mathcal D}^{\bot}_{\rho}, 
	\end{align*}
	where $\hat{\mathcal D}_{\rho}^{\bot}$ denotes the annihilator of the distribution $\hat{\mathcal D}_{\rho}$. It follows that for all $z$ around $z_a$,
	\begin{align}\label{Eq:hu}
		\begin{array}{l}
			\left\langle {dh^u(z)},ad^i_fg^u(z) \right\rangle=0,  \ \ 0\le i \le  \rho-2,  \quad	\left\langle {dh^u(z)},ad^{\rho-1}_fg^u(z) \right\rangle\ne 0;   \quad\quad
			\left\langle {dh^u(z)},ad^i_fg^v(z) \right\rangle=0,  \ \ 0\le i \le  \rho-1.\\
		\end{array}
	\end{align} 
	Recall the following result \cite{Isidori:1995:NCS:545735}\cite{nijmeijer1990nonlinear}:
	\begin{align}\label{Eq:Leibniz}
		\begin{array}{ll}
			\left\langle {dh(z)},ad^i_fg(z) \right\rangle=0,  \  0\le i \le l-2   \Rightarrow  \left\langle {dh(z)},ad^{l-1}_fg(z) \right\rangle=(-1)^i\left\langle {dL^i_fh(z)},ad^{l-1-i}_fg(z) \right\rangle, \  0\le i \le l-1,
		\end{array}
	\end{align}
	where $h(z)$ is a scalar function, $f(z)$ and $g(z)$ are vector fields.
	
	It can be deduced from (\ref{Eq:hu}) and (\ref{Eq:Leibniz}) that for all $z$ around $z_a$,
	\begin{align}\label{Eq:hu1}
		\begin{array}{l}
			\left\langle {dL^i_fh^u(z)},ad^j_fg^u(z) \right\rangle=0, \   0\le i \le  \rho-2,\  0\le j\le \rho-i-2,	\quad
			\left\langle {dL^i_fh^u(z)},ad^{\rho-i-1}_fg^u(z) \right\rangle\ne 0, \   0\le i \le  \rho-2 ;  \\
			\left\langle {dL^i_fh^u(z)},ad^j_fg^v(z) \right\rangle=0, \ \ 0\le i \le  \rho-1,\  0\le j\le \rho-i-1.  
		\end{array}
	\end{align}
	By using (\ref{Eq:hu1}), we have the following table for the expressions of 
	$\left\langle {\rd L^i_fh^u},ad^{j}_fg^u \right\rangle$, $0\le i\le \rho-\bar \rho$, $\bar\rho-1\le j\le\rho-1$:
	$$
	\begin{smallmatrix}
		&ad^{\bar\rho-1}_fg^u& ad^{\bar\rho}_fg^u& \cdots& ad^{\rho-1}_fg^u\\
		\rd h^u&0&0&\cdots&\left\langle {\rd h^u},ad^{\rho-1}_fg^u \right\rangle\\
		\cdots&\cdots&\cdots&*&\\ 
		\rd L^{\rho-\bar\rho-1}_fh^u&0&\left\langle {\rd L^{\rho-\bar\rho-1}_fh^u},ad^{\bar\rho}_fg^u \right\rangle&  &\\
		\rd L^{\rho-\bar\rho}_fh^u& \left\langle {\rd L^{\rho-\bar\rho}_fh^u},ad^{\bar\rho-1}_fg^u \right\rangle&&  &?\\
	\end{smallmatrix}
	$$
	Notice that all the anti-diagonal elements of the above table are  nonzero by (\ref{Eq:hu1}).  It follows that the co-distribution  $$\Omega_1=\Span\left\lbrace dL^{i}_fh^u, \ \ 0\le i\le \rho-\bar \rho    \right\rbrace$$ is of dimension $\rho-\bar \rho+1$ around $z_a$. We have
	$\Omega_1 \subseteq \mathcal D^{\bot}_{\bar \rho-1} $ because 	%$$\mathcal D_{\bar \rho-1}=\Span\left\lbrace g^u,\dots, ad^{\bar \rho-2}_fg^u, g^v,\dots, ad^{\bar \rho-2}_fg^v \right\rbrace $$ and  
	$$ 
	\begin{aligned}
		\left\langle \rd L^i_fh^u(z) ,ad^{j}_fg^u(z) \right\rangle \overset{(\ref{Eq:hu1})}{=}  0,\ \ 0\le i\le\rho-\bar \rho, \ \ 0\le j\le \bar\rho-2,\\ \left\langle \rd L^i_fh^u(z) ,ad^{j}_fg^v(z) \right\rangle \overset{(\ref{Eq:hu1})}{=}  0, \ \ 0\le i\le\rho-\bar \rho, \ \ 0\le j\le \bar\rho-2. 
	\end{aligned}
	$$		
	It is seen that $\dim \mathcal{D}^{\bot}_{\bar \rho-1}-\dim \Omega_1=(n^*-(2\bar\rho-2))-(\rho-\bar\rho+1)=1$ and $\Omega_1 \subsetneq \mathcal D^{\bot}_{\bar \rho-1} $. Then by  the involutivity of ${\mathcal D}_{\bar \rho-1}$ (condition (FL3)), we can choose a scalar function $h^v(z)$ such that 
	\begin{align*}
		\Span\{\rd h^v\}+ \Omega_1=  {\mathcal D}^{\bot}_{\bar\rho-1},
	\end{align*}	
	which implies that for all $z$ around $z_a$,
	\begin{align}\label{Eq:hv}
		\begin{array}{l}
			\left\langle {dh^v(z)},ad^i_fg^u(z) \right\rangle=0,  ~ 0\le i \le  \bar \rho-2; \quad 
			\left\langle {dh^v(z)},ad^i_fg^v(z) \right\rangle=0,  ~ 0\le i \le  \bar \rho-2,\quad \left\langle {dh^v(z)},ad^{\bar \rho-1}_fg^v(z) \right\rangle\neq 0. 
		\end{array}
	\end{align} 
	It can be deduced by (\ref{Eq:hv}) and (\ref{Eq:Leibniz}) that for all $z$ around $z_a$,
	\begin{align}\label{Eq:hv1} 
		\begin{array}{l}
			\left\langle {dL^i_fh^v(z)},ad^j_fg^u(z) \right\rangle=0, ~~	 
			0\le i \le  \bar\rho-2,\ 0\le j\le \bar\rho-i-2 ;  \\
			\left\langle {dL^i_fh^v(z)},ad^j_fg^v(z) \right\rangle=0, ~0\le i \le  \bar\rho-2,\ \ 0\le j\le \bar\rho-i-2,  \quad
			\left\langle {dL^i_fh^v(z)},ad^{\bar\rho-i-1}_fg^v(z) \right\rangle\ne 0, ~ 0\le i \le  \bar\rho-2.  
		\end{array}
	\end{align}
	By using (\ref{Eq:hu1}) and (\ref{Eq:hv1}), we can construct the following table:  
		\begin{equation} \label{Eq:table1}
			\begin{smallmatrix}
				&g^v&g^u&\cdots&\cdots&ad^{\bar\rho-1}_fg^v&ad^{\bar\rho-1}_fg^u& ad^{\bar\rho}_fg^u& \cdots& ad^{\rho-1}_fg^u\\
				\rd h^u&0&0&\cdots&\cdots&0&0&0&\cdots&\left\langle {\rd h^u},ad^{\rho-1}_fg^u \right\rangle\\
				\cdots&\cdots&\cdots&\cdots&\cdots&\cdots&\cdots&\cdots&*&\\ 
				\rd L^{\rho-\bar\rho-1}_fh^u& 0&0&\cdots&\cdots&0&0&\left\langle {\rd L^{\rho-\bar\rho-1}_fh^u},ad^{\bar\rho}_fg^u \right\rangle&  &\\
				\rd L^{\rho-\bar\rho}_fh^u&0 &0&\cdots&\cdots&0&\left\langle {\rd L^{\rho-\bar\rho}_fh^u},ad^{\bar\rho-1}_fg^u \right\rangle&&  &?\\
				\rd h^v&0& 0&\cdots&\cdots&\left\langle {\rd h^v},ad^{\bar\rho-1}_fg^v \right\rangle&?&&  &\\
				\cdots& 0&0&\cdots&*~~~&  &\\
				\cdots& 0&0&*&?~~~&  &\\
				\rd L^{\rho-1}_fh^u&0&   L_{g^u }L^{\rho-1}_fh^u~~~~~~  &\\
				\rd L^{\bar \rho-1}_fh^v& L_{g^v} L^{\bar \rho-1}_f h^v  & ?  ~~~&?&&  &? 
			\end{smallmatrix}
		\end{equation} 
	%$$
	%\begin{smallmatrix}
	%&g^v&g^u&ad_fg^v&ad_fg^u&\cdots&\cdots&ad^{\bar\rho-1}_fg^v&ad^{\bar\rho-1}_fg^u& ad^{\bar\rho}_fg^u& \cdots& ad^{\rho-1}_fg^u\\
	%\rd h^u&0&0&0&0&\cdots&\cdots&0&0&0&0&\left\langle {\rd h^u},ad^{\rho-1}_fg^u \right\rangle\\
	%\cdots&\cdots&\cdots&\cdots&\cdots&\cdots&\cdots&\cdots&\cdots&\cdots&*&\\ 
	%\rd L^{\rho-\bar\rho-1}_fh^u&0&0&0&0&\cdots&\cdots&0&0&\left\langle {\rd L^{\rho-\bar\rho-1}_fh^u},ad^{\bar\rho}_fg^u \right\rangle&  &\\
	%\rd L^{\rho-\bar\rho}_fh^u&0&0&0&0&\cdots&\cdots&0&\left\langle {\rd L^{\rho-\bar\rho}_fh^u},ad^{\bar\rho-1}_fg^u \right\rangle&&  &\\
	%\rd h^v&0&0&0&0&\cdots&\cdots&\left\langle {\rd h^v},ad^{\bar\rho-1}_fg^u \right\rangle&\left\langle {\rd h^v},ad^{\bar\rho-1}_fg^u \right\rangle&&  &\\
	%\cdots&0&0&0&0&\cdots&*~~~&  &\\
	%\cdots&0&0&0&0&*&*~~~&  &\\
	%\rd L^{\rho-2}_fh^u&0&0&0&\left\langle {\rd L^{\rho-2}_fh^u},ad_fg^u \right\rangle~~~~~~  &\\
	%\rd L^{\rho-2}_fh^v&0&0&\left\langle {\rd L^{\rho-2}_fh^v},ad^{\bar\rho-1}_fg^u \right\rangle&\left\langle {\rd L^{\rho-2}_f h^v},ad_fg^u \right\rangle~~~~~~  &\\
	%\rd L^{\rho-1}_fh^u&0&\left\langle {\rd L^{\rho-1}_fh^u},g^u \right\rangle~~~~~~&&  &\\
	%\rd L^{\rho-1}_fh^v&\left\langle {\rd L^{\rho-1}_f h^v},g^u \right\rangle&\left\langle {\rd L^{\rho-1}_f h^v},ad_fg^u \right\rangle~~~~~~&&  &\\
	%\end{smallmatrix}
	%$$	
	Notice that all the anti-diagonal elements of table (\ref{Eq:table1}) are  nonzero. 	It follows  that the $(\rho+\bar\rho)\times (\rho+\bar\rho)=n^*\times n^*$ matrix 
	$$
	\frac{\partial \psi}{\partial z}(z)\left[\begin{matrix}
		g^v&g^u&\cdots&\cdots&ad^{\bar\rho-1}_fg^v&ad^{\bar\rho-1}_fg^u& ad^{\bar\rho}_fg^u& \cdots& ad^{\rho-1}_fg^u
	\end{matrix} \right](z)
	$$
	is invertible around $z_a$, where 
	\begin{align}\label{Eq:psi}
		\psi=(h^u,\dots,L^{\rho-1}_fh^u,h^v,\dots,L^{\bar \rho-1}_fh^v).
	\end{align}
	Thus the Jacobian matrix $\frac{\partial \psi(z)}{\partial z}$  is invertible around $z_a$ and $\psi$ is a local diffeomorphism. Then set
	\begin{align}\label{Eq:tfunc}
		\begin{array}{c}
			a^u(z) =L^{\rho}_fh^u(z), ~~~~ b^u(z)=L_{g^u}L^{\rho-1}_fh^u(z), ~~~~ a^v(z) =L^{\bar \rho}_fh^v(z),~~~~ b^v(z)=L_{g^v}L^{\bar \rho-1}_fh^v(z),~~~~ \tilde\lambda(z) =L_{g^u}L^{\bar \rho-1}_fh^v(z). 
		\end{array} 
	\end{align}
	Note that $b^u(z)$ and $b^v(z)$ are nonzero at $z_p$. It is seen that   $\Sigma^{u^*v^*}$ is mapped, via the coordinates transformations $\xi=(\xi_1,\xi_2)=\psi(z)$ and the feedback transformation (\ref{Eq:twofbDACS2}), into the Brunovsk{\'y} form $\Sigma^w_{Br}=\Sigma^{w^*}_{Br}$ of (\ref{Eq:bnskf}) with indices $\rho$ and $\bar \rho$. 
	
	Case 2: If $\rho<\bar \rho$, then we have $
	\mathcal D_0\subsetneq \hat{\mathcal D}_{1}\subsetneq  \dots \subsetneq \hat{\mathcal D}_{\rho}\subsetneq \mathcal D_{\rho}\subsetneq\hat{\mathcal D}_{ \rho+1}= {\mathcal D}_{ \rho+1}\subsetneq  \dots  = {\mathcal D}_{\bar \rho-1}\subsetneq\hat{\mathcal D}_{\bar \rho}= {\mathcal D}_{\bar\rho}=\hat{\mathcal D}_{\bar\rho+j}= {\mathcal D}_{\bar\rho+j}$, $ j>0$. It follows that $\hat{\mathcal D}_{\bar\rho}= {\mathcal D}_{\bar \rho}=\hat{\mathcal D}_{n^*}={\mathcal D}_{n^*}$. Then by (FL2) of Theorem \ref{Thm:IFL}, we have $\hat{\mathcal D}_{\bar \rho}=TM^*$ and thus $\dim \hat{\mathcal D}_{\bar \rho}=n^*$. 
	%by a dimensional argument, we can deduce  $\rho+\bar \rho=n^*$. 
	By ${\mathcal D}_{\bar\rho-1}\subsetneq \hat{\mathcal D}_{\bar\rho}$ and  (\ref{Eq:dim}), we have $\dim {\mathcal D}_{\bar\rho-1}=n^*-1$.
	Now by the involutivity of ${\mathcal D}_{\bar\rho}$   (condition (FL1)), we can choose a scalar function  $h^v(z)$ such that 
	\begin{align*}
		{\rm span}\left\lbrace dh^v\right\rbrace={\mathcal D}^{\bot}_{\bar\rho-1}. 
	\end{align*}
	Then following a similar proof as in Case 1, we can show that the  distribution
	$$\Omega_2=\Span\left\lbrace dL^{i}_fh^v, \ \ 0\le i\le \bar\rho- \rho  -1  \right\rbrace$$
	is of dimension $\rho-\bar\rho$ around $z_a$ and $\Omega_2\subsetneq \hat{\mathcal D}^{\bot}_{\rho}$. Notice that $\dim \hat{\mathcal D}^{\bot}_{\rho}=n^*-(2\rho-1)=\bar\rho -\rho+1$, we have $\dim \hat{\mathcal D}^{\bot}_{\rho}-\dim \Omega_2=1$. Thus by   the involutivity of $\hat{\mathcal D}_{\rho}$ (condition (FL2)), we can choose a scalar function $h^u(z)$ such that 
	\begin{align*}
		\Span\{\rd h^u\}+ \Omega_2=  \hat{\mathcal D}^{\bot}_{\rho}.
	\end{align*}	
	Then, similarly as in Case 1,  we   construct the following table:
 $$
		\begin{smallmatrix}
			&g^v&g^u&\cdots&\cdots&ad^{\rho-1}_fg^v&ad^{\rho-1}_fg^u& ad^{ \rho}_fg^v& \cdots& ad^{\bar\rho-1}_fg^v\\
			\rd h^v&0&0&\cdots&\cdots&0&0&0&\cdots&\left\langle {\rd h^v},ad^{\bar \rho-1}_fg^v \right\rangle\\
			\cdots&\cdots&\cdots&\cdots&\cdots&\cdots&\cdots&\cdots&*&?\\ 
			\rd L^{\bar \rho-\rho-1}_fh^v& 0&0&\cdots&\cdots&0&0&\left\langle {\rd L^{\bar \rho-\rho-1}_fh^v},ad^{\rho}_fg^v \right\rangle&  &\\
			\rd h^u&0& 0&\cdots&\cdots&0&\left\langle {\rd h^u},ad^{\rho-1}_fg^u \right\rangle&&  &\\
			\rd L^{\bar \rho-\rho}_fh^v&0 &0&\cdots&\cdots&\left\langle {\rd L^{\bar \rho-\rho}_fh^v},ad^{\rho-1}_fg^v \right\rangle~~~  &?&&&?\\
			\cdots& 0&0&\cdots&*~~~&  &\\
			\cdots& 0&0&*~~~&&  &\\
			\rd L^{\rho-1}_fh^u&0&   L_{g^u }L^{\rho-1}_fh^u~~~~~~  &\\
			\rd L^{\bar \rho-1}_fh^v& L_{g^v} L^{\bar\rho-1}_f h^v  & ?  &?&&  &?&&&?\\
		\end{smallmatrix}
		$$ 
	and show that all the anti-diagonal elements of the table are nonzero around $z_a$. Finally, we define a diffeomorphism $\psi$ and  functions $a^u$, $b^u$, $a^v$, $b^v$, $\tilde \lambda$ in the same form as   (\ref{Eq:psi}) and (\ref{Eq:tfunc}) of Case 1. 
	It is seen that   $\Sigma^{uv}$ can also be transformed into the Brunovsk{\'y} form $\Sigma^w_{Br}=\Sigma^{w^*}_{Br}$ of (\ref{Eq:bnskf}) with indices $\rho$ and $\bar \rho$ via the change of coordinates $\xi=\psi(z)$ and the feedback transformation (\ref{Eq:twofbDACS2}).  
\end{proof}

\section*{Acknowledgements} 
The author wants to thank  Witold Respondek (INSA de Rouen) and Stephan Trenn (University of Groningen) for several helpful discussions and suggestions.

%\section*{Conflict of interest} 
%The authors declare that they have no conflict of interest. 
 
\bibliographystyle{WileyNJD-AMA}
\bibliography{bibNFs} 

\begin{thebibliography}{10}
\providecommand \doibase [0]{http://dx.doi.org/}%

\bibitem{lewis1986survey}
Lewis FL. A survey of linear singular systems. {\it Circuits, Systems and
  Signal Processing} 1986\string; 5(1)\string: 3--36.

\bibitem{lewis1992tutorial}
Lewis FL. A tutorial on the geometric analysis of linear time-invariant
  implicit systems. {\it Automatica} 1992\string; 28(1)\string: 119--137.

\bibitem{dai1989singular}
Dai L. {\it Singular Control Systems}. 118.
\newblock Springer .
\newblock 1989.

\bibitem{loiseau1991feedback}
Loiseau JJ, {\"O}z{\c{c}}aldiran K, Malabre M, Karcanias N. Feedback canonical
  forms of singular systems. {\it Kybernetika} 1991\string; 27(4)\string:
  289--305.

\bibitem{lebret1994proportional}
Lebret G, Loiseau JJ. Proportional and proportional-derivative canonical forms
  for descriptor systems with outputs. {\it Automatica} 1994\string;
  30(5)\string: 847--864.

\bibitem{chen2021geometric}
Chen Y, Respondek W. Geometric analysis of differential-algebraic equations via
  linear control theory. {\it SIAM Journal on Control and Optimization}
  2021\string; 59(1)\string: 103--130.

\bibitem{berger2013controllability}
Berger T, Reis T. Controllability of linear differential-algebraic systems--a
  survey. In:  {\it Surveys in Differential-Algebraic Equations I}Springer 2013
  (pp. 1--61).

\bibitem{cobb1984controllability}
Cobb D. Controllability, observability, and duality in singular systems. {\it
  IEEE Transactions on Automatic Control} 1984\string; 29(12)\string:
  1076--1082.

\bibitem{frankowska1990controllability}
Frankowska H. On controllability and observability of implicit systems. {\it
  Systems \& Control Letters} 1990\string; 14(3)\string: 219--225.

\bibitem{geerts1993invariant}
Geerts T. Invariant subspaces and invertibility properties for singular
  systems: The general case. {\it Linear algebra and its Applications}
  1993\string; 183\string: 61--88.

\bibitem{ozccaldiran1986geometric}
{\"O}z{\c{c}}aldiran K. A geometric characterization of the reachable and the
  controllable subspaces of descriptor systems. {\it Circuits, Systems and
  Signal Processing} 1986\string; 5\string: 37--48.

\bibitem{rabier2000nonholonomic}
Rabier PJ, Rheinboldt WC. {\it Nonholonomic Motion of Rigid Mechanical Systems
  from a DAE Viewpoint}. 68.
\newblock Society for Industrial and Applied Mathematics .
\newblock 2000.

\bibitem{Kumar1995}
Kumar A, Daoutidis P. Feedback control of nonlinear
  differential-algebraic-equation systems. {\it AIChE Journal} 1995\string;
  41(3)\string: 619--636.

\bibitem{riaza2008differential}
Riaza R. {\it Differential-Algebraic Systems: Analytical Aspects and Circuit
  Applications}.
\newblock World Scientific .
\newblock 2008.

\bibitem{chen2021nfs}
Chen Y, Trenn S, Respondek W. Normal forms and internal regularization of
  nonlinear differential-algebraic control systems. {\it International Journal
  of Robust and Nonlinear Control} 2021.
\newblock In press, doi:\url{https://doi.org/10.1002/rnc.5623}.

\bibitem{respondek1985geometric}
Respondek W. Geometric methods in linearization of control systems. {\it Banach
  Center Publications} 1985\string; 1(14)\string: 453--467.

\bibitem{tall2005feedback}
Tall I, Respondek W. Feedback equivalence of nonlinear control systems: a
  survey on formal approach. {\it Chaos in Automatic Control} 2005\string:
  156-281.

\bibitem{nijmeijer1990nonlinear}
Nijmeijer H, Schaft V.~dA. {\it Nonlinear Dynamical Control Systems}. 175.
\newblock Springer .
\newblock 1990.

\bibitem{Isidori:1995:NCS:545735}
Isidori A. {\it Nonlinear Control Systems}.
\newblock Secaucus, NJ, USA: Springer-Verlag New York, Inc.
\newblock 3rd~ed. 1995.

\bibitem{brockett1978feedback}
Brockett RW. Feedback invariants for nonlinear systems. {\it IFAC Proceedings
  Volumes} 1978\string; 11(1)\string: 1115--1120.

\bibitem{Jakubczyk&Respondek1980}
Jakubczyk B, Respondek W. On linearization of control systems. {\it Bull. Acad.
  Polonaise Sci. Ser. Sci. Math.} 1980\string: 517-522.

\bibitem{SU198248}
Su R. On the linear equivalents of nonlinear systems. {\it Systems Control
  Letters} 1982\string; 2(1)\string: 48 - 52.

\bibitem{hunt1983global}
Hunt L, Su R, Meyer G. Global transformations of nonlinear systems. {\it IEEE
  Transactions on Automatic Control} 1983\string; 28(1)\string: 24--31.

\bibitem{4793293}
Xiaoping L. On linearization of nonlinear singular control systems. In:  {\it
  American Control Conference}IEEE. ; 1993\string: 2284-2287.

\bibitem{kawaji1994feedback}
Kawaji S, Taha EZ. Feedback linearization of a class of nonlinear descriptor
  systems. In:  {\it Proceedings of the 33rd IEEE Conference on Decision and
  Control}. 4. IEEE. ; 1994\string: 4035--4037.

\bibitem{983833}
Wang J, Chen C. Exact linearization of nonlinear differential algebraic
  systems. In:  {\it International Conferences on Info-Tech and Info-Net.
  Proceedings}. 4. ; 2001\string: 284-290 vol.4.

\bibitem{chen2019internal}
Chen Y, Respondek W. Internal and external linearization of semi-explicit
  differential algebraic equations. {\it IFAC-PapersOnLine} 2019\string;
  52(16)\string: 292--297.

\bibitem{chen2021from}
Chen Y, Respondek W. From {Morse} triangular form of {ODE} control systems to
  feedback canonical form of {DAE} control systems. Submitted to publish,
  preprint available from \url{https://arxiv.org/abs/2103.14913};  2021.

\bibitem{lee2001introduction}
Lee JM. {\it Introduction to {Smooth} {Manifolds}}.
\newblock Springer .
\newblock 2001.

\bibitem{reich1991existence}
Reich S. On an existence and uniqueness theory for nonlinear
  differential-algebraic equations. {\it Circuits, Systems and Signal
  Processing} 1991\string; 10(3)\string: 343--359.

\bibitem{rabier1994geometric}
Rabier PJ, Rheinboldt WC. A geometric treatment of implicit
  differential-algebraic equations. {\it Journal of Differential Equations}
  1994\string; 109(1)\string: 110--146.

\bibitem{berger2016controlled}
Berger T. Controlled invariance for nonlinear differential--algebraic systems.
  {\it Automatica} 2016\string; 64\string: 226--233.

\bibitem{chen2021ADHS}
Chen Y, Trenn S. A singular perturbed system approximation of nonlinear
  diffrential-algebraic equations. Accepted by IFAC conference of ADHS2021,
  preprint available from \url{https://arxiv.org/abs/2103.12146};  2021.

\bibitem{chen2019geometric}
Chen Y. {\it {Geometric Analysis of Differential-Algebraic Equations and
  Control Systems: Linear, Nonlinear and Linearizable}}. PhD thesis. Normandie
  Universit{\'e},  2019.
\newblock Available from
  \url{https://tel.archives-ouvertes.fr/tel-02478957/document}.

\bibitem{brunovsky1970classification}
Brunovsk{\`y} P. A classification of linear controllable systems. {\it
  Kybernetika} 1970\string; 6(3)\string: 173--188.

\bibitem{Hunt&Su1981}
Hunt L, Su R. Linear equivalents of nonlinear time-varying systems. In:  {\it
  Proc. Int. Symposium on Math. Theory of Networks and Systems}Santa Monica. ;
  1981\string: 119--123.

\bibitem{Berger2016zero}
Berger T. The zero dynamics form for nonlinear differential-algebraic systems.
  {\it IEEE Transactions on Automatic Control} 2017\string; 62(8)\string:
  4131-4137.

\bibitem{arai1998nonholonomic}
Arai H, Tanie K, Shiroma N. Nonholonomic control of a {three-DOF} planar
  underactuated manipulator. {\it IEEE Transactions on Robotics and Automation}
  1998\string; 14(5)\string: 681--695.

\end{thebibliography}

%\clearpage

%\section*{Author Biography}
%
%\begin{biography}{\includegraphics[width=66pt,height=86pt,draft]{empty}}{\textbf{Author Name.} This is sample author biography text this is sample author biography text this is sample author biography text this is sample author biography text this is sample author biography text this is sample author biography text this is sample author biography text this is sample author biography text this is sample author biography text this is sample author biography text this is sample author biography text this is sample author biography text this is sample author biography text this is sample author biography text this is sample author biography text this is sample author biography text this is sample author biography text this is sample author biography text this is sample author biography text this is sample author biography text this is sample author biography text.}
%\end{biography}

\end{document}